\documentclass{amsart}

\usepackage[english]{babel}

\usepackage{amsmath,amsfonts,amssymb,latexsym,graphics}
\usepackage[bookmarks]{hyperref}
\usepackage[dvips]{color}
\usepackage[pdftex]{graphicx}
\usepackage{fullpage}
\usepackage{pstricks}

\newtheorem{prop}{Proposition}
\newtheorem{lem}{Lemma}
\newtheorem{definition}{Definition}
\newtheorem{theorem}{Theorem}
\newtheorem{theo}{Theorem}
\newtheorem{corollary}{Corollary}

\newcommand{\qbin}[2]{\genfrac{[}{]}{0pt}{}{#1}{#2}_q}
\newcommand{\tqbin}[2]{\genfrac{[}{]}{0pt}{1}{#1}{#2}_q}
\newcommand{\acc}[2]{\genfrac{\{}{\}}{0pt}{}{#1}{#2}}
\newcommand{\tacc}[2]{\genfrac{\{}{\}}{0pt}{1}{#1}{#2}}

\newcommand{\ket}[1]{\ensuremath{|#1\rangle}}
\newcommand{\bra}[1]{\ensuremath{\langle #1|}}

\renewenvironment{itemize}{\begin{list}{\labelitemi}{\leftmargin=1.5em}}{\end{list}}

\begin{document}

\date{\today}
\title{Rook placements in Young diagrams and permutation enumeration}
\author{Matthieu Josuat-Verg\`es}
\address{LRI, CNRS and Universit\'e Paris-Sud, B\^atiment 490, 91405 Orsay, FRANCE}
\email{josuat@lri.fr}

\begin{abstract}
Given two operators $\hat D$ and $\hat E$ subject to the relation 
$\hat D\hat E -q \hat E \hat D =p$, and a 
word $w$ in $M$ and $N$, the rewriting of $w$ in normal form is combinatorially
described by rook placements in a Young diagram. We give enumerative results 
about these rook placements, particularly in the case where $p=(1-q)/q^2$. This
case naturally arises in the context of the PASEP, a random process whose 
partition function and stationary distribution are expressed using two operators 
$D$ and $E$ subject to the relation $DE-qED=D+E$ (matrix Ansatz). Using the link 
obtained by Corteel and Williams between the PASEP, permutation tableaux and
permutations, we prove a conjecture of Corteel and Rubey about permutation 
enumeration. This result gives the generating function for permutations of given 
size with respect to the number of ascents and occurrences of the pattern 13-2, this 
is also the moments of the $q$-Laguerre orthogonal polynomials.
\end{abstract}

\maketitle

\section{Introduction}

In the recent work of Postnikov \cite{AP}, permutations appear with a 
new description, as pattern-avoiding fillings of Young diagrams. More 
precisely, he made a correspondence between positive Grassmann cells, 
these pattern-avoiding fillings called \hbox{\rotatedown{$\Gamma$}}-diagrams, 
and decorated permutations (which are permutations with a weight 2 on each fixed 
point). In particular, 
the usual permutations are in bijection with permutation tableaux, a subset of \hbox{\rotatedown{$\Gamma$}}-diagrams. Permutation tableaux have then 
been studied by Steingr\`imsson, Williams, Burstein, Corteel, Nadeau 
\cite{Bu,CN,CW,SW}, and revealed themselves very useful in the combinatorics of
permutations.

\bigskip

Corteel and Williams observed, and explained, a rather surprising link 
between these permutation tableaux and the stationary distribution of a 
classical process of statistical physics, the Partially Asymmetric Self-Exclusion 
Process (PASEP). This model is described in \cite{CW,DEHP}. More precisely, 
the stationary probability of a given state in the process is 
proportional to 
the sum of weights of permutation tableaux of a given shape. The factor behind
this proportionality is the partition function, which is the sum of weights of
permutation tableaux of a given half-perimeter.

\bigskip

Another way of finding the stationary distribution of the PASEP is the 
matrix Ansatz of Derrida, Evans, Hakim, and Pasquier\cite{DEHP}. Suppose that we 
have operators $D$ and $E$, a row vector $\bra{W}$ and a column vector $\ket{V}$ 
such that
\[ DE-qED = D+E, \qquad \bra{W}E = \bra{W}, \qquad D\ket{V} = \ket{V}, 
\qquad \hbox{and}
\quad \bra{W}\ket{V}=1. \]
Then, coding any state of the process by a word $w$ of length $n$ in $D$ and $E$,
the stationary probability of the state $w$ is given by 
$\bra{W}w\ket{V} \bra{W}(D+E)^n\ket{V}^{-1}$.
This denominator $\bra{W}(D+E)^n\ket{V}$ is the partition function.

\bigskip

We briefly describe how the matrix Ansatz is related to permutation tableaux
\cite{CW}. First, notice that
there are unique polynomials $n_{i,j}\in \mathbb{Z}[q]$ such that
\[ (D+E)^n = \sum_{i,j\geq 0} n_{i,j} E^i D^j. \]
This sum is called the {\it normal form} of $(D+E)^n$. It is particularly useful,
since for example the sum of the coefficients $n_{i,j}$ give an evaluation of
$\bra{W}(D+E)^n\ket{V}$. If $D$ and $E$ would commute, the expansion of $(D+E)^n$
would be described by binomial coefficients. But in this non-commutative 
context, the process of expanding and rewriting $(D+E)^n$ in normal form is
combinatorially described by permutation tableaux. Then each coefficient 
$n_{i,j}$ is a generating function for permutation tableaux satisfying
certain conditions. Equivalently this can be done with the 
{\it alternative tableaux} defined by Viennot \cite{XGV}.

\bigskip

One of the ideas at the origin of this article is the following. From $D$ and 
$E$ of the matrix Ansatz, we define new operators
\[ \hat D = \frac{q-1}q D + \frac 1q \qquad \hbox{and} 
\qquad \hat E = \frac{q-1}q E + \frac 1q.\]
Some immediate consequences are

\begin{equation} \label{comrel}
\hat D \hat E - q \hat E \hat D = \frac{1-q}{q^2}, \qquad \bra{W} \hat E = \bra{W}, 
\qquad \hbox{and}\quad \hat D\ket{V} = \ket{V}. 
\end{equation}
This new commutation relation is in a way much more simple than the one satisfied
by $D$ and $E$. It is close to the relation between creation and annihilation operators
classicaly studied in quantum physics.
Moreover, from these definitions we have $q(y\hat D + \hat E) + (1-q)(yD+E) = 1+y$
for some parameter $y$.
By isolating one term of the left-hand side and raising to the $n$ with the 
binomial rule, we get the following inversion formulas between $(yD+E)^n$ and 
$(y\hat D+\hat E)^n$:

\begin{equation} (1-q)^n (yD+E)^n = \sum_{k=0}^n \binom n k (1+y)^{n-k} (-1)^k  
q^k(y\hat D+ \hat E)^k,\quad\hbox{and}
\label{inv1}
\end{equation}
\begin{equation}
q^n (y\hat D+\hat E)^n = \sum_{k=0}^n \binom n k (1+y)^{n-k} (-1)^k  (1-q)^k
(yD+E)^k.
\label{inv2}
\end{equation}
In particular, the first formula means that if we want to compute the 
coefficients of the normal form of $(yD+E)^n$, it is enough to compute the ones
of $(y\hat D + \hat E)^n$ for all $n$. Notice that taking the normal form is a linear 
operation.

\bigskip

Up to a factor $-q$, these operators $\hat D$ and $\hat E$ are also defined in 
\cite{USW} and \cite{BECE}.
In the first reference, Uchiyama, Sasamoto and Wadati use the new
relation between $\hat D$ and $\hat E$ to find explicit matrix representations of
these operators. They derive the eigenvalues and eigenvectors of $\hat D+\hat E$, 
and consequently the ones of $D+E$, in terms of orthogonal polynomials.
In the second reference, Blythe, Evans, Colaiori and Essler also use these 
eigenvalues and obtain an integral form for $\bra{W}(D+E)^n\ket{V}$. They also provide
an exact integral-free formula of this quantity, although quite complicated since 
it contains three sum signs and several $q$-binomial coefficients.

\bigskip

In this article, instead of working on representations of $\hat D$ and $\hat E$
and their eigenvalues, we study the combinatorics of the rewriting in normal
form of $(\hat D+\hat E)^n$, and more generally $(y\hat D+\hat E)^n$ for some 
parameter $y$. 
In the case of $\hat D$ and $\hat E$, the objects that appear are the
{\it rook placements in Young diagrams}, long-known  
since the results of Kaplansky, Riordan, Goldman, Foata and Sch\"utzenberger
(see \cite{RPS} and references therein). This method is described in \cite{AV}, 
and is the same as the one leading to permutation tableaux or alternative 
tableaux in the case of $D$ and $E$.

\medskip

\begin{definition} Let $\lambda$ be a Young diagram. A {\it rook placement} of 
shape $\lambda$ is a partial filling of the cells of $\lambda$ with rooks 
(denoted by a circle $\circ$), such that there is at most one rook per row 
(resp. per column).
\end{definition}

\smallskip

For convenience, we distinguish with a cross ($\times$) each cell of the Young 
diagram that is not below (in the same column) or to the left (in the same row) 
of a rook (see Figures \ref{alpha},\ref{beta} and \ref{phi} further). 
We will see that the number of crosses is an important statistic on rook placements.
This statistic was introduced in \cite{GR}, as a generalization of the inversion
number for permutations. Indeed, if $\lambda$ is a square of side length $n$, a
rook placements $R$ with $n$ rooks may be seen as the graph of a permutation 
$\sigma\in\mathfrak{S}_n$, and then the number of crosses in $R$ is the 
inversion number of $\sigma$.

\smallskip

\begin{definition}
The weight of a rook placement $R$ with $r$ rooks and $s$ crosses is 
$w(R) = p^r q^s$.
\end{definition}

\bigskip

The enumeration of rook placements leads to an evaluation of 
$\bra{W} (y\hat D+\hat E)^{n-1} \ket{V}$, hence an evaluation 
of $\bra{W} (yD+E)^{n-1} \ket{V}$ via the inversion formula (\ref{inv1}).
This is the main result of this article:

\smallskip

\begin{theorem} For any $n>0$, we have \label{main}
\[
  \bra{W} (yD+E)^{n-1} \ket{V} = \tfrac 1{y(1-q)^n} \sum_{k=0}^n (-1)^k 
  \left(\sum_{j=0}^{n-k} y^j\Big( \tbinom{n}{j}\tbinom{n}{j+k} - 
   \tbinom{n}{j-1}\tbinom{n}{j+k+1}\Big)\right)
  \left(\sum_{i=0}^k y^iq^{i(k+1-i)}  \right).
\]
\end{theorem}

The combinatorial interpretation of this polynomial, in terms of permutations,
is given in Proposition \ref{comb}. When $y=1$, this can be specialized to:

\begin{theorem} For any $n>0$, we have \label{main2}
\[
  \bra{W}(D+E)^{n-1}\ket{V} = \frac 1{(1-q)^n} \sum_{k=0}^n (-1)^k 
  \left( \binom{2n}{n-k} - \binom{2n}{n-k-2} \right)
  \Bigg( \sum_{i=0}^k q^{i(k+1-i)} \Bigg). 
  \label{cr}
\]
\end{theorem}

We can see Theorem \ref{main2} as a variation of the
Touchard-Riordan formula \cite{JT}. This classical formula gives the $q$-enumeration 
of fixed-point-free involutions of size $2n$ with respect to the number 
of crossings, and it is also the $2n$th moment of the $q$-Hermite polynomials.
This formula is:
\begin{equation} \label{tou_rio}
\sum_{I \in \hbox{Inv}(2n,0)} q^{\hbox{cr}(I)} = \frac 1 {(1-q)^n}
\sum_{k=0}^n (-1)^k \left( \binom{2n}{n-k} - \binom{2n}{n-k-1} \right)
q^{\frac {k(k+1)}2},
\end{equation}
where Inv$(2n,0)$ is the set of fixed-point-free involutions on $2n$ elements, 
and where the number of crossings cr$(I)$ is given in Definition \ref{defstat}.

\bigskip

These two theorems were conjectured by Corteel and Rubey. The earliest conjecture,
when $y=1$ and here stated as Theorem \ref{main2}, was first proved by Rubey and 
Prellberg \cite{TP} in May 2008. The same method also can be used to give an 
alternative proof of our Theorem \ref{main}. This alternative proof, as well as 
the material of this article, are summarized in the extended abstract \cite{CJPR}.

\bigskip

This alternative proof relies on a decomposition of weighted Motzkin paths, which
gives a combinatorial explanation of the factor $\sum_{j=0}^{n-k} y^j\big( \tbinom{n}{j}\tbinom{n}{j+k} - \tbinom{n}{j-1}\tbinom{n}{j+k+1}\big)$. But on the other hand,
the factor $\sum_{i=0}^k y^iq^{i(k+1-i)}$ is obtained by solving a functional 
equation and this is a completely non-combinatorial step. It may be possible to use 
the involution principle instead of a functional equation to obtain $\sum_{i=0}^k y^i
q^{i(k+1-i)}$ but this is still an open problem at the time of writing.

\bigskip

Besides references earlier mentioned, we have to point out the previous results 
of Williams \cite{LW}, where Corollary 6.3 gives the coefficients of $y^m$ in
$\bra{W}(yD+E)^n\ket{V}$. It was obtained by a more direct approach, via the 
enumeration of \hbox{\rotatedown{$\Gamma$}}-diagrams, and was the only known 
polynomial formula for the distribution of a permutation pattern of length 
greater than 2 (see Proposition \ref{comb}). Whereas Williams's work is rather 
focused on \hbox{\rotatedown{$\Gamma$}}-diagrams, our results give more simple 
formulas in the case of permutation tableaux and permutations. Moreover 
Williams's formulas have also been obtained by Kasraoui, Stanton and Zeng in 
their work on orthogonal polynomials \cite{KSZ}.

\bigskip

This article is organised as follows.
In Section \ref{sec2}, we describe the link between rook placements and the rewriting
of $(\hat D+\hat E)^n$ in normal form. In Sections \ref{sec3}, \ref{sec4}, \ref{sec5}, 
we obtain enumerative results about rook placements, in particular Section \ref{sec4} 
contains the bijective step of this enumeration. In Section \ref{sec6}, we use these 
results to prove Theorem \ref{main}, give the combinatorial interpretation of
$\bra{W}(yD+E)^n\ket{V}$ and some applications of the main theorem. In an appendix
we give a combinatorial proof of Proposition \ref{T0knprop}, which gives a 
generalization of the Touchard-Riordan formula.

\bigskip

\section*{Acknowledgement}

\medskip

This work is partially supported by a Research Fellowship from the Erwin 
Schr\"odinger International Institute for Mathematical Physics in Vienna.

\bigskip

I want to thank Lauren Williams for her interesting discussions and suggestions, 
and my advisor Sylvie Corteel for her strong support through this work and many
improvements of this article. I address particular thanks to Martin Rubey for
suggesting this problem, and to Jeremy Lovejoy and Dan Drake for some corrections
in this article.

\bigskip

\section*{Notations and conventions}

\label{secn}

We denote by Par$(n-k,k)$ the set of Young diagrams with
exactly $k$ rows and $n-k$ columns, allowing empty rows and columns. The
integer $n$ is the half-perimeter of the diagram $\lambda\in$ Par$(n-k,k)$, and
we can see $\lambda$ as an integer partition $(\lambda_1,\ldots,\lambda_k)$
with $n-k\geq\lambda_1\geq \ldots \geq \lambda_k\geq 0$.
We use the French convention. We denote by $|\lambda|$ the number of cells
in $\lambda$, which is also $\sum\lambda_i$.

\smallskip

The North-East boundary of $\lambda\in$ Par$(n-k,k)$ is a path of 
$n$ steps, $k$ of them being vertical and $n-k$ horizontal. 
Reciprocally, for any word $w$ of length $n$ in $\hat D$ and $\hat E$, 
with $k$ occurrences of $\hat E$, we define $\lambda(w)\in$ Par$(n-k,k)$ by
the following rule: we read $w$ from left to right, and draw one step East
for each factor $\hat D$, and one step South for each factor $\hat E$.

\medskip

We denote by Inv$(n,k)$ the set of involutions on $\{1,\ldots,n\}$ with $k$
fixed points.

\medskip

We use the classical $q$-analogs of integers, factorials, and binomial coefficients:
\[ 
  [n]_q = \frac{1-q^n}{1-q}, \qquad [n]_q! = \prod_{i=1}^n [i]_q, \qquad 
  \hbox{and} \quad \qbin n k= \frac{[n]_q!}{[k]_q! [n-k]_q!}.
\]

\begin{prop} \cite{RPS} The $q$-binomial coefficient has the following combinatorial
interpretation:
\[
   \qbin n k = \sum_{\lambda\in \hbox{Par}(n-k,k)} q^{|\lambda|}, \qquad
   \qquad \hbox{and}\qquad
   q^{k(k+1)/2}\qbin n k = \sum_{\substack{ \lambda\in \hbox{Par}(n,k) \\
   \lambda \hbox{ has distinct} \\ \hbox{non-zero parts}  }} 
   q^{|\lambda|}.
\]
\label{part}
\end{prop}

\begin{definition} For any $k,n\geq 0$, the Delannoy numbers are defined by
\[ \acc n k = \binom n k - \binom n {k-1}. \]
\end{definition}

\begin{prop} When $2k\leq n$, 
the number $\acc n k$ counts the left factors of Dyck paths of $n$
steps ending at height $n-2k$. In particular, $ \acc {2n} n$ is the $n$th 
Catalan number. They satisfy the relations:
\[ 
  \acc n k = \acc{n-1} k + \acc{n-1} {k-1}, 
  \qquad \quad \acc n {n+1-k} = - \acc n k, 
\]
\[ 
  \acc 0 0 = \acc 0 1 = 1, \qquad\hbox{ and } \acc n k=0 
  \hbox{ if } k\notin \{0,\ldots,n+1\}.
\]
\end{prop}

\begin{proof} The number of left factors of Dyck paths of $n$ steps ending at 
height $n-2k$ is easily seen to satisfy the same relations as $\acc n k$: we 
just have to distinguish two cases whether the last step is going up or down.
\end{proof}

\bigskip

\section{From operator relations to rook placements}

\label{sec2}

\medskip

In this section, we make the link between the coefficients of the normal form 
of $(\hat D+ \hat E)^n$, and rook placements in Young diagrams. This is done
via a combinatorial description of the rewriting in normal form. When $q=1$, we
can view it as a combinatorial statement of a classical result in statistical
physics, called Wick's theorem. The principle of this method is the same 
as the one described in the introduction, making the link between $D$ and $E$
and permutation tableaux. Moreover the results of these section are presented
in \cite{AV} in a slightly different form.

\medskip

From now on we assume that $\hat D$ and $\hat E$ are such that $\hat D\hat E
-q\hat E\hat D = p$ for some parameter $p$, which is a slight generalization
of the relation (\ref{comrel}).
As in the case of $D$ and $E$, any word $w$ in $\hat D$ and $\hat E$ can be
uniquely written in normal form:
\[ w = \sum_{i,j\geq 0} c_{i,j}(w) \hat E^i \hat D^j,  \]
where $c_{i,j}(w)\in\mathbb{Z}[p,q]$.
We have: 
\[ \bra{W}w\ket{V} = \sum_{i,j\geq 0} c_{i,j}(w). \]
The combinatorial interpretation of this polynomial is given by the following 
proposition:

\smallskip

\begin{prop} \label{rewriting}
Let $w$ be a word in $\hat E$ and $\hat D$. Then $\bra{W}w\ket{V}$ is 
the sum of weights of rook placements of shape $\lambda(w)$.
\end{prop}

\begin{proof} Let us denote by $T_w$ the sum of weights of rook placements
of shape $\lambda(w)$. We prove with a recurrence on $|\lambda(w)|$, that
$T_w=\bra{W}w\ket{V}$.

\medskip

The base case, $|\lambda(w)|=0$, is the situation where the word $w$ is already 
in normal form: $w=\hat E^i\hat D^j$ for some $i$ and $j$. So we directly have
$\bra{W}w\ket{V}=1$ from the properties of $\bra{W}$ and $\ket{V}$ given in (\ref{comrel}).
This 1 corresponds to the 
unique rook placement of shape $\lambda(w)$, which contains no rook and no cross 
since there is no cell in this diagram.

\medskip

Now we assume that $|\lambda(w)|>0$. It is possible to factorize $w$ into
$w=w_1\hat D\hat E w_2$. Indeed, this factor $\hat D \hat E$ corresponds to a 
corner of $\lambda(w)$, and there is at least one corner since $|\lambda(w)|>0$.
The commutation relation of $\hat D$ and $\hat E$ gives:
\[
  w = qw_1\hat E\hat Dw_2 + pw_1w_2, \quad \hbox{hence} \quad
  \bra{W}w\ket{V} = q\bra{W}w_1\hat E\hat Dw_2\ket{V} + p\bra{W}w_1w_2\ket{V}.
\]
So it remains to show that the same relation holds for $T_w$:
\[ T_w = qT_{w_1\hat E \hat D w_2} + pT_{w_1w_2}. \]
To this end, we distinguish two kinds of rook placements of shape $\lambda(w)$, 
whether the corner corresponding to the factor $\hat D\hat E$
contains a cross or a rook. It cannot be empty, since being a corner there 
is no cell above it or to its right that may contain a rook.

\smallskip

The sum of weights of rook placements of the first type is 
$qT_{w_1\hat E \hat D w_2}$. Indeed, when removing the corner the
weight is divided by $q$ and we can get any rook placement of shape 
$\lambda(w_1\hat D \hat E w_2)$.

\smallskip

Similarly, the sum of weights of rook placements of the second type is 
$pT_{w_1w_2}$. Indeed, when removing the corner, its row and its column,
the weight is divided by $p$ and we can get any rook placement of shape 
$\lambda(w_1w_2)$.
\end{proof}

\smallskip

Since $(\hat D + \hat E)^n$ expands into the sum of all words of length $n$ in 
$\hat D$ and $\hat E$, we also obtain:

\smallskip

\begin{prop} For any $n$, $\bra{W}(\hat D + \hat E)^n\ket{V}$ is the sum
of weights of rook placements of half-perimeter $n$.
\end{prop}

\smallskip

We can also expand $(y\hat D + \hat E)^n$ and get the sum over all words of 
length $n$ in $\hat D$ and $\hat E$. But in this case each word $w$ has a 
coefficient $y^m$, where $m$ is the number of occurrences of $\hat D$ in $w$. 
Via the correspondence between words and Young diagrams, the number of occurrences
of $\hat D$ in $w$ is the number of columns in $\lambda(w)$. This leads to a 
refined version of the previous proposition.

\smallskip

\begin{prop} 
For any $n$, $\bra{W}(y\hat D + \hat E)^n\ket{V}$ is the generating
function for rook placements of half-perimeter $n$, the parameter $y$ 
counting the number of columns. \label{rewriting2}
\end{prop}

\bigskip

\section{Basic results about rook placements}

\label{sec3}
In this section we introduce the recurrence relation which will be used in the
enumeration of rook placements, and we present two simple examples of 
enumeration. These two examples involve the $q$-binomial coefficients and the 
Delannoy numbers defined at the end of the introduction, and they introduce
the more general formulas we will show later.

\smallskip

\begin{definition} \label{defstat}
Let $T_{j,k,n}(p,q)$ be the sum of weights of rook placements
of half-perimeter $n$, with $k$ rows, and with $j$ rows containing no rook (or
equivalently, with $k-j$ rooks). We also define:
\[ T_{k,n}(p,q) = \sum_{j=0}^k T_{j,k,n}(p,q), \quad  \hbox{and} \quad
T_n(p,q,y) = \sum_{k=0}^n y^k T_{k,n}(p,q).\]
So $T_{k,n}(p,q)$ is the sum of weights of rook placements of half-perimeter 
$n$ with $k$ rows, and $T_n(p,q,y)$ is the generating function of rook placements 
of half-perimeter $n$, the parameter $y$ counting the number of rows.
\end{definition}

\smallskip

Since there is an obvious transposition-symmetry, we can also view the parameter
$y$ as counting the number of columns.
These are polynomials in the variables $p$, $q$ and $y$, so we will
sometimes omit the arguments. From Proposition \ref{rewriting2} we know that 
$T_n(p,q,y)$ is equal to $\bra{W}(y\hat D+\hat E)^n\ket{V}$. In Figure \ref{firstval} 
we give some examples of these polynomials.

\begin{figure}[h!tp]\psset{unit=3mm}
\[ T_{0,1,3}=pq+p+p, \hspace{2cm}
T_{1,1,3}=1+q+q^2, \hspace{2cm}
T_2=1+(1+q+p)y +y^2.
\]

\begin{pspicture}(8,3)
\psframe(0,0)(2,1) \psline(1,0)(1,1) \rput(0.5,0.5){\small $\circ$}
\rput(1.5,0.5){\small $\times$}
\psframe(3,0)(5,1) \psline(4,0)(4,1) \rput(4.5,0.5){\small $\circ$}
\psframe(6,0)(7,1) \psline(7,0)(8,0) \rput(6.5,0.5){\small $\circ$}
\end{pspicture}
\hspace{2.3cm}
\begin{pspicture}(8,3)
\psline(0,1)(0,0)(2,0)
\psframe(3,0)(4,1)\psline(4,0)(5,0)\rput(3.5,0.5){\small $\times$}
\psframe(6,0)(8,1) \psline(7,0)(7,1) \rput(6.5,0.5){\small $\times$}
\rput(7.5,0.5){\small $\times$}
\end{pspicture}
\hspace{2.4cm}
\begin{pspicture}(0,0)(8,3)
\psline(0,0)(0,2)
\psline(1,1)(1,0)(2,0)
\psframe(3,0)(4,1)\rput(3.5,0.5){\small $\times$}
\psframe(5,0)(6,1)\rput(5.5,0.5){\small $\circ$}
\psline(7,0)(9,0)
\end{pspicture}
\caption{\label{firstval} Some small values of $T_{j,k,n}$ and $T_n$, together
with the rook placements corresponding to each term.}
\end{figure}

\begin{prop} \label{Trecprop} We have the following recurrence relation:
\begin{equation}
T_{j,k,n} = T_{j-1,k-1,n-1} + q^jT_{j,k,n-1} + p[j+1]_qT_{j+1,k,n-1}. \label{Trec}
\end{equation}
\end{prop}

\begin{proof} We can distinguish three kinds of rook placements enumerated by 
$T_{j,k,n}$ (see Figure \ref{three_rec}):
\begin{itemize}
\item the first column is of size strictly less than $k$,
\item the first column is of size $k$ and contains no rook,
\item or the first column is of size $k$ and contains exactly one rook.
\end{itemize}
We show that these three types respectively lead to the three terms of the 
recurrence relation.

\begin{figure}[h!tp]\psset{unit=3mm}
\begin{pspicture}(0,0)(5,5)
\psline[linewidth=0.6mm](0,5)(0,0)(1,0)(1,4)(0,4)
\psline(1,0)(5,0)(5,1) \psline(1,4)(2,4) \rput(3.5,3){$\ddots$}
\end{pspicture}\hspace{1cm}
\begin{pspicture}(0,0)(5,5)
\psframe[linewidth=0.6mm](0,0)(1,5)
\psline(1,0)(5,0)(5,1) \psline(1,4)(2,4) \rput(3.5,3){$\ddots$}
\end{pspicture}\hspace{1cm}
\begin{pspicture}(0,0)(5,5)
\psframe[linewidth=0.6mm](0,0)(1,5)
\psline(1,0)(5,0)(5,1) \psline(1,4)(2,4) \rput(3.5,3){$\ddots$}
\psarc[linewidth=0.45mm](0.5,2.5){0.2}{0}{360}
\end{pspicture}
\caption{The three kinds of rook placements \label{three_rec} we distinguish
for proving Proposition \ref{Trecprop}.}
\end{figure}
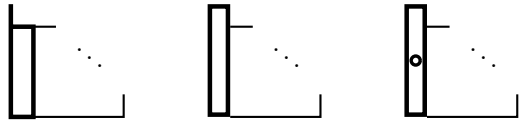

\bigskip

The first case is the situation where the first step of the North-East boundary
is a step down, or equivalently the first row is of size 0. Removing this step 
(or row) is a bijection between these first-type rook placements, and the ones 
enumerated by $T_{j-1,k-1,n-1}$, the first term of (\ref{Trec}).

\smallskip

In the second case, the first column contains exactly $j$ crosses, one per row
without rook. So removing the first column is a bijection between the second-type 
rook placements, and the ones enumerated by $T_{j,k,n-1}$, and this bijection 
changes the weight by a factor $q^j$. This explains the second term of 
(\ref{Trec}).

\smallskip

In the third case, removing the first column is not a bijection since there 
are several possibilities for the position of the rook in this column. But 
this map has the property that for any $R$ enumerated by $T_{j+1,k,n-1}$, 
the preimage set of $R$ contains $j+1$ elements, and their weights are $pw(R)$, 
$pqw(R)$, $\dots$, $pq^jw(R)$. See Figure \ref{firstcol} for an example.
This shows that the sum of weights of the 
third-type rook placements is the third term of (\ref{Trec}), and completes 
the proof.
\end{proof}

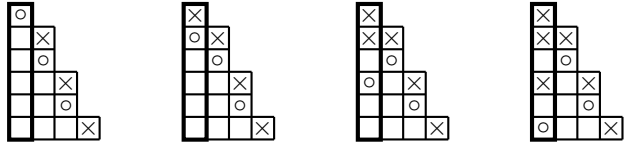
\begin{figure}[h!tp] \psset{unit=3mm}
\begin{pspicture}(0,0)(4,6)
\psline(0,0)(4,0)\psline(0,1)(4,1)\psline(0,2)(3,2)\psline(0,3)(3,3)
\psline(0,4)(2,4)\psline(0,5)(2,5)\psline(0,6)(1,6)\psline(0,0)(0,6)
\psline(1,0)(1,6)\psline(2,0)(2,5)\psline(3,0)(3,3)\psline(4,0)(4,1)
\psframe[linewidth=0.55mm](-0.1,-0.1)(1.1,6.1)
\rput(2.5,1.5){$\circ$}\rput(2.5,2.5){$\times$}
\rput(1.5,3.5){$\circ$}\rput(1.5,4.5){$\times$}\rput(3.5,0.5){$\times$}
\rput(0.5,5.5){$\circ$}
\end{pspicture}\hspace{1cm}
\begin{pspicture}(0,0)(4,6)
\psline(0,0)(4,0)\psline(0,1)(4,1)\psline(0,2)(3,2)\psline(0,3)(3,3)
\psline(0,4)(2,4)\psline(0,5)(2,5)\psline(0,6)(1,6)\psline(0,0)(0,6)
\psline(1,0)(1,6)\psline(2,0)(2,5)\psline(3,0)(3,3)\psline(4,0)(4,1)
\psframe[linewidth=0.55mm](-0.1,-0.1)(1.1,6.1)
\rput(2.5,1.5){$\circ$}\rput(2.5,2.5){$\times$}
\rput(1.5,3.5){$\circ$}\rput(1.5,4.5){$\times$}\rput(3.5,0.5){$\times$}
\rput(0.5,5.5){$\times$}\rput(0.5,4.5){$\circ$}
\end{pspicture}\hspace{1cm}
\begin{pspicture}(0,0)(4,6)
\psline(0,0)(4,0)\psline(0,1)(4,1)\psline(0,2)(3,2)\psline(0,3)(3,3)
\psline(0,4)(2,4)\psline(0,5)(2,5)\psline(0,6)(1,6)\psline(0,0)(0,6)
\psline(1,0)(1,6)\psline(2,0)(2,5)\psline(3,0)(3,3)\psline(4,0)(4,1)
\psframe[linewidth=0.55mm](-0.1,-0.1)(1.1,6.1)
\rput(2.5,1.5){$\circ$}\rput(2.5,2.5){$\times$}
\rput(1.5,3.5){$\circ$}\rput(1.5,4.5){$\times$}\rput(3.5,0.5){$\times$}
\rput(0.5,5.5){$\times$}\rput(0.5,4.5){$\times$}\rput(0.5,2.5){$\circ$}
\end{pspicture}\hspace{1cm}
\begin{pspicture}(0,0)(4,6)
\psline(0,0)(4,0)\psline(0,1)(4,1)\psline(0,2)(3,2)\psline(0,3)(3,3)
\psline(0,4)(2,4)\psline(0,5)(2,5)\psline(0,6)(1,6)\psline(0,0)(0,6)
\psline(1,0)(1,6)\psline(2,0)(2,5)\psline(3,0)(3,3)\psline(4,0)(4,1)
\psframe[linewidth=0.55mm](-0.1,-0.1)(1.1,6.1)
\rput(2.5,1.5){$\circ$}\rput(2.5,2.5){$\times$}
\rput(1.5,3.5){$\circ$}\rput(1.5,4.5){$\times$}\rput(3.5,0.5){$\times$}
\rput(0.5,5.5){$\times$}\rput(0.5,4.5){$\times$}\rput(0.5,2.5){$\times$}
\rput(0.5,0.5){$\circ$}
\end{pspicture}
\caption{ We have here four rook placements of the third type, which are equal
when we remove the first column. Here we have $n=10$, $k=6$, $j=3$, and the sum
of their weights is $(p+pq+pq^2+pq^3)p^2q^3=p[j+1]_q(p^2q^3)$. This illustrates
the third term of \ref{Trec}.
  \label{firstcol}}
\end{figure}

\begin{prop} For any $k,n$ we have:
\begin{equation}
T_{k,k,n} = \qbin n k.
\end{equation}
\end{prop}

\begin{proof}
We are counting rook placements without any rook, {\it i.e.} such that all 
cells contain a cross. So this is a direct application of Proposition \ref{part}.
\end{proof}

This proposition is illustrated for example in Figure \ref{firstval} where we 
see that $T_{1,1,3}=1+q+q^2=[3]_q$. The second example of this section is more 
subtle and we begin with the following lemma.

\smallskip

\begin{lem} Given a Young diagram $\lambda$, the number of rook placements of 
shape $\lambda$ having no cross and exactly one rook per row is either 0 or 1.
It is 1 in the case where the North-East boundary is a Dyck path (which means
that the $i$th row of $\lambda$ starting from the top contains at least $i$ 
cells, for any $i$ between 1 and the number of rows).
\end{lem}

\begin{proof} Suppose that $R$ is a rook placement with no cross and exactly
one rook per row. Then the $i$ first rows contain $i$ rooks, which are
necessarly in $i$ different columns. So the $i$th row contain at least $i$
cells. this is true for any $i$, so the North-East boundary is a Dyck path.

\smallskip

It remains to prove that there is a unique such rook placement in the case where 
the North-East boundary of a Young diagram $\lambda$ is a Dyck path. We show that
there is only one way to build this rook placement starting from a empty diagram
$\lambda$. First, notice that each corner of the 
diagram must contain a rook (as we saw in previous section, the general statement
is that each corner contains either a rook or a cross). Then, if we consider the 
subdiagram of cells that are not in the same row or column of these rooks (see
Figure \ref{dyck}), again all corners of this subdiagram must contain a rook by 
the same argument. We can even say that his North-East boundary is also a Dyck path: 
indeed, the boundary of the subdiagram is obtained from the boundary of the diagram
by removing each occurrence of a step right followed by a step down. So we can 
conclude by recurrence.

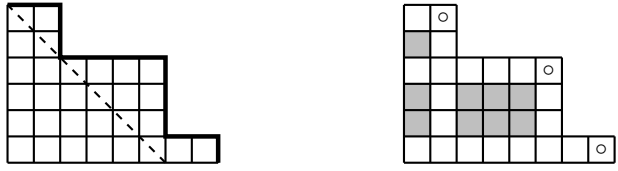
\begin{figure}[h!tp] \label{dyck}
\mbox{
\psset{unit=0.35cm}
\begin{pspicture}(0,0)(8,6)
\psline(0,0)(8,0)  \psline(0,0)(0,6)
\psline(0,1)(8,1)  \psline(1,0)(1,6)
\psline(0,2)(6,2)  \psline(2,0)(2,6)
\psline(0,3)(6,3)  \psline(3,0)(3,4)
\psline(0,4)(6,4)  \psline(4,0)(4,4)
\psline(0,5)(2,5)  \psline(5,0)(5,4)
\psline(0,6)(2,6)  \psline(6,0)(6,4)
                   \psline(7,0)(7,1)
                   \psline(8,0)(8,1)
\psline[linestyle=dashed,dash=1mm 1mm](6,0)(0,6)
\psline[linewidth=0.6mm](0,6)(2,6)(2,4)(6,4)(6,1)(8,1)(8,0)
\end{pspicture}
}
\hspace{2cm}
\mbox{
\psset{unit=0.35cm}
\begin{pspicture}(0,0)(8,6)
\psframe[fillstyle=solid, fillcolor=lightgray, linewidth=0](0,4)(1,5)
\psframe[fillstyle=solid, fillcolor=lightgray, linewidth=0](0,1)(1,3)
\psframe[fillstyle=solid, fillcolor=lightgray, linewidth=0](2,1)(5,3)
\psline(0,0)(8,0)  \psline(0,0)(0,6)
\psline(0,1)(8,1)  \psline(1,0)(1,6)
\psline(0,2)(6,2)  \psline(2,0)(2,6)
\psline(0,3)(6,3)  \psline(3,0)(3,4)
\psline(0,4)(6,4)  \psline(4,0)(4,4)
\psline(0,5)(2,5)  \psline(5,0)(5,4)
\psline(0,6)(2,6)  \psline(6,0)(6,4)
                   \psline(7,0)(7,1)
                   \psline(8,0)(8,1)
\rput(1.5,5.5){\small $\circ$}
\rput(5.5,3.5){\small $\circ$}
\rput(7.5,0.5){\small $\circ$}
\end{pspicture}
}
\caption{Example of Young diagram whose North-East boundary is a Dyck path, 
{\it i.e.} doesn't go below the dotted line. The number of rows is $k=6$, the 
half-perimeter is $n=14$, and the path ends at height $n-2k=2$.}
\end{figure}
\end{proof}

\begin{prop} If $2k<n$, we have:
\begin{equation} 
T_{0,k,n}(p,0) = p^k \acc n k.
\end{equation}
\end{prop}

\begin{proof} We are counting rook placements with no cross (since $q=0$ here)
and exactly $k$ rooks. Each of these rook placements has weight $p^k$, so
we just have to prove that there are $\acc n k$ such rook placements.
Knowing that $\acc n k$ is the number of left factors of Dyck paths of $n$ steps 
ending at height $n-2k$, this is a consequence of the previous lemma.
\end{proof}

\bigskip

\section{Rook placements and involutions}

\label{sec4}
\bigskip

In this section we present the bijective step of the enumeration of rook 
placements. Indeed, the recurrence (\ref{Trec}) is rather complicated
to be solved directly. But thanks to this bijective step, we show that
there is a simple relation between $T_{j,k,n}$ and $T_{0,k-j,n}$,
and also that there is a simple recurrence relation satisfied by
$T_{0,k,n}$.

\bigskip

Given a rook placement $R$ of half-perimeter $n$, we define an involution 
$\alpha(R)$ by the following construction. We label the North-East boundary of $R$ with 
integers from 1 to $n$, as shown in the left part of Figure \ref{alpha}. This
way, each column and each row has a label between 1 and $n$. If a column, or a row,
is labelled with $i$ and does not contain a rook, then $i$ is a fixed point of
$\alpha(R)$. If there is a rook at the intersection of column $i$ and
row $j$, then the involution $\alpha(R)$ sends $i$ to $j$ (and $j$ to $i$).

\bigskip

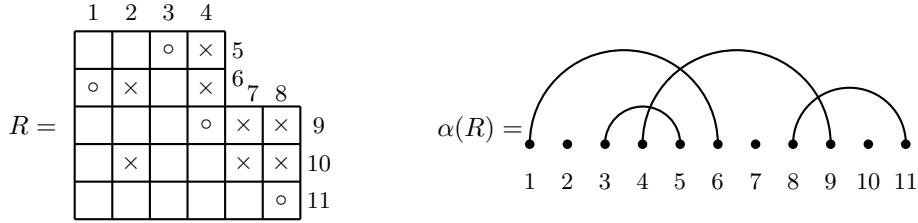
\begin{figure}[h!tp]
\psset{unit=0.5cm}
\mbox{
\psset{unit=0.5cm}
\begin{pspicture}(-1,0)(7,5)
\rput(-1,2.5){$R=$ }
\psline(0,0)(6,0)  \psline(0,0)(0,5)
\psline(0,1)(6,1)  \psline(1,0)(1,5)
\psline(0,2)(6,2)  \psline(2,0)(2,5)
\psline(0,3)(6,3)  \psline(3,0)(3,5)
\psline(0,4)(4,4)  \psline(4,0)(4,5)
\psline(0,5)(4,5)  \psline(5,0)(5,3)\psline(6,0)(6,3)
\rput(0.5,5.5){\small 1}
\rput(1.5,5.5){\small 2}
\rput(2.5,5.5){\small 3}
\rput(3.5,5.5){\small 4}
\rput(4.35,4.5){\small 5}
\rput(4.35,3.75){\small 6}
\rput(4.75,3.35){\small 7}
\rput(5.5,3.35){\small 8}
\rput(6.5,2.5){\small 9}
\rput(6.5,1.5){\small 10}
\rput(6.5,0.5){\small 11}
\rput(0.5,3.5){\small $\circ$}
    \rput(1.5,3.5){\small $\times$}
    \rput(3.5,3.5){\small $\times$}
    \rput(1.5,1.5){\small $\times$}
\rput(2.5,4.5){\small $\circ$}
    \rput(3.5,4.5){\small $\times$}
\rput(3.5,2.5){\small $\circ$}
    \rput(4.5,1.5){\small $\times$}
    \rput(4.5,2.5){\small $\times$}
    \rput(5.5,1.5){\small $\times$}
    \rput(5.5,2.5){\small $\times$}
\rput(5.5,0.5){\small $\circ$}
\end{pspicture}
}
\hspace{1.6cm}
\begin{pspicture}(0.8,-2)(12,2.5)
\rput(0.8,0.4){$\alpha(R)=$ }
\psdots(2,0)(3,0)(4,0)(5,0)(6,0)(7,0)(8,0)(9,0)(10,0)(11,0)(12,0)
\psarc(4.5,0){2.5}{0}{180} \psarc(5,0){1}{0}{180} \psarc(7.5,0){2.5}{0}{180}
\psarc(10.5,0){1.5}{0}{180}
\rput(2,-1){\small 1}
\rput(3,-1){\small 2}
\rput(4,-1){\small 3}
\rput(5,-1){\small 4}
\rput(6,-1){\small 5}
\rput(7,-1){\small 6}
\rput(8,-1){\small 7}
\rput(9,-1){\small 8}
\rput(10,-1){\small 9}
\rput(11,-1){\small 10}
\rput(12,-1){\small 11}
\end{pspicture}
\caption{Example of a rook placement and its image by the map $\alpha$.
\label{alpha} }
\end{figure}
\bigskip

Given a rook placement $R$ of half-perimeter $n$, we also define a Young 
diagram $\beta(R)$ by the following construction. The North-East boundary of $R$
is a sequence of East steps and South steps, respectively denoted by $S$ and $E$.
In this sequence, we overline the letter $E$ if it corresponds to a East step on top
of a column containing a rook. Similarly, we overline the letter $S$ if it corresponds
to a South step on the right of a row containing a rook. 
The non-overlined letters give a subword of this sequence, and $\beta(R)$ is the 
Young diagram corresponding to this sequence.

For example, with $R$ given in Figure \ref{alpha}, the North-East boundary is 
$EEEESSEESSS$. With the overlined letters we obtain 
$\overline{E}E\overline{E}\overline{E}\overline{S}\overline{S}E\overline{E}
\overline{S}S\overline{S}$. The non-overlined letters give the word $EES$. So 
in this case $\beta(R)$ is the rectangular Young diagram with 1 row and 2 columns.
See Figure \ref{beta} for another example.


\begin{figure}[htp]\psset{unit=4mm}
\begin{pspicture}(0,0)(5,4) 
\rput(-1.3,2){$R=$}
\psframe[fillstyle=solid,fillcolor=lightgray,linewidth=0](1,0)(2,4)
\psframe[fillstyle=solid,fillcolor=lightgray,linewidth=0](0,0)(5,1)
\psframe[fillstyle=solid,fillcolor=lightgray,linewidth=0](0,2)(4,3)
\psframe[fillstyle=solid,fillcolor=lightgray,linewidth=0](3,1)(4,2)
\psline(0,0)(5,0)  \psline(0,0)(0,4)
\psline(0,1)(5,1)  \psline(1,0)(1,4)
\psline(0,2)(5,2)  \psline(2,0)(2,4)
\psline(0,3)(4,3)  \psline(3,0)(3,4)
\psline(0,4)(3,4)  \psline(4,0)(4,3)
                   \psline(5,0)(5,2)
\rput(0.5,1.5){\small $\times$} \rput(0.5,3.5){\small $\times$}
\rput(1.5,2.5){\small $\circ$}  \rput(1.5,3.5){\small $\times$}
\rput(2.5,1.5){\small $\times$} \rput(2.5,2.5){\small $\times$}
\rput(2.5,3.5){\small $\times$} \rput(3.5,0.5){\small $\circ$}
\rput(3.5,1.5){\small $\times$} \rput(3.5,2.5){\small $\times$}
\rput(4.5,0.5){\small $\times$} \rput(4.5,1.5){\small $\times$}
\end{pspicture}
\hspace{2cm}
\begin{pspicture}(0,-1)(3,2)
\rput(-2,1){$\beta(R)=$}
\psline(0,0)(3,0)\psline(0,0)(0,2)
\psline(0,1)(3,1)\psline(1,0)(1,2)
\psline(0,2)(2,2)\psline(2,0)(2,2)
                 \psline(3,0)(3,1)
\end{pspicture}
\caption{Example of a rook placement and its image by the map $\beta$.
The grey cells are the ones that are in the same row or the same column as a rook.
The shape of the rook placement is given by the sequence 
$E\overline{E}ES\overline{E}\overline{S}ES\overline{S}$. So $\beta(R)$ is defined
by the sequence $EESES$. \label{beta} }
\end{figure}
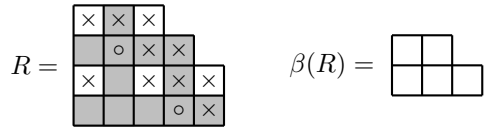

\noindent
{\it Remark:} We can see in Figure \ref{beta} that the cells of
$\beta(R)$ are obtained from $R$ by removing all gray cells. This is a general
fact, and $|\beta(R)|$ is the number of cells in $R$ with no rook in the same row 
and no rook in the same column.
\bigskip

We also define $\phi(R) = (\alpha(R),\beta(R))$. See Figure \ref{phi} for another
example of a rook placement and its images by the maps $\alpha$ and $\beta$.

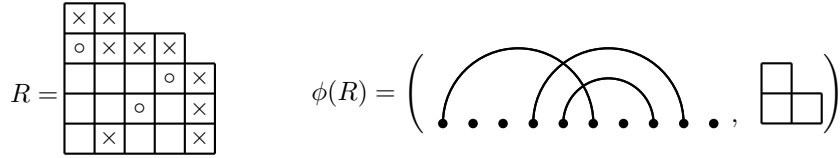
\begin{figure}[htp]\psset{unit=4mm}
\begin{pspicture}(-1,0)(5,5) 
\rput(-1,2){$R=$}
\psline(0,0)(5,0)  \psline(0,0)(0,5)
\psline(0,1)(5,1)  \psline(1,0)(1,5)
\psline(0,2)(5,2)  \psline(2,0)(2,5)
\psline(0,3)(5,3)  \psline(3,0)(3,4)
\psline(0,4)(4,4)  \psline(4,0)(4,4)
\psline(0,5)(2,5)  \psline(5,0)(5,3)
\rput(0.5,4.5){\small $\times$} \rput(1.5,4.5){\small $\times$}
\rput(0.5,3.5){\small $\circ$}  \rput(1.5,3.5){\small $\times$}
\rput(2.5,3.5){\small $\times$} \rput(3.5,3.5){\small $\times$}
\rput(3.5,2.5){\small $\circ$}  \rput(4.5,2.5){\small $\times$}
\rput(2.5,1.5){\small $\circ$}  \rput(4.5,1.5){\small $\times$}
\rput(1.5,0.5){\small $\times$}  \rput(4.5,0.5){\small $\times$}
\end{pspicture}
\hspace{1.8cm}
\begin{pspicture}(-1.5,-1)(10,2.5)
\rput(-1.5,1){$\phi(R)=\Bigg($}
\psdots(1,0)(2,0)(3,0)(4,0)(5,0)(6,0)(7,0)(8,0)(9,0)(10,0) 
\psarc(3.5,0){2.5}{0}{180}
\psarc(3.5,0){2.5}{0}{180}
\psarc(6.5,0){2.5}{0}{180}
\psarc(6.5,0){2.5}{0}{180}
\psarc(6.5,0){1.5}{0}{180} 
\psarc(6.5,0){1.5}{0}{180}
\rput(10.7,0){,}
\end{pspicture}
\hspace{0.4cm}
\begin{pspicture}(0,-1)(2.2,2)
\psline(0,0)(2,0)\psline(0,0)(0,2)
\psline(0,1)(2,1)\psline(1,0)(1,2)
\psline(0,2)(1,2)\psline(2,0)(2,1)
\rput(2.4,1){$\Bigg)$}
\end{pspicture}
\caption{Example of a rook placement and its image by the map $\phi$\label{phi}.}
\end{figure}

\bigskip

\begin{prop} The map $\phi$ is a bijection between rook placements in Young 
diagrams of half-perimeter $n$, and couples $(I,\lambda)$ where $I$ is an 
involution on $\{1,\dots,n\}$ and $\lambda$ a Young diagram of half-perimeter
$|$Fix$(I)|$. If $\phi(R) = (I,\lambda)$, the number of rows 
(resp. columns) of $\lambda$ is equal to the number of rows (resp. columns) 
without a rook in $R$.
\end{prop}

This is a classical argument so we don't give a complete proof. This bijection 
was already defined in 
\cite{Ke}, in terms of partial involutions, {\it i.e.} involutions on subsets
of $\{1,\dots,n\}$. These partial involutions are
equivalent to the couples $(I,\lambda)$ in the sense that they are
involutions with a weight 2 for each fixed point.

Indeed, a couple $(I,\lambda)$ may be seen as an involution with two kinds of fixed 
points: those corresponding to vertical steps in $\lambda$ and those corresponding 
to horizontal steps. Similarly, a partial involution $I$ on 
$\{1,\dots,n\}$ may also be seen as an involution on $\{1,\dots,n\}$ with 
two kinds of fixed points: the ones that are not in the domain of $I$, and
the ones that are in the domain and fixed by $I$.

\bigskip

Now that we have built a bijection, it remains to describe how the weight of a rook
placement reads in the couple $(I,\lambda)$. We need the following definitions.

\medskip

\begin{definition} For any involution $I$, we call
\begin{itemize}
\item
an arch of $I$, a couple (i,j) such that 
$i<j$ and $I(i)=j$,
\item
a crossing of $I$, a pair of arches $((i,j),(k,l))$ such that 
$i<k<j<l$,
\item
the height of a fixed point $k\in$Fix$(I)$ the number of arches $(i,j)$ such 
that $i<k<j$.
\end{itemize}
We denote by $cr(I)$ the number of crossings of $I$, and by $ht(k)$ the height
of the fixed point $k$.
\end{definition}

\smallskip

For example, let us consider the involution $\alpha(R)$ in Figure \ref{phi}.
There are two crossings, $((1,6),(4,9))$ and $((1,6),(5,8))$. The fixed points are 
Fix$(I)=\{2,3,7,10\}$ and their respective heights are $1$, $1$, $2$ 
and $0$.

\smallskip

\begin{prop} \label{crosses}
Let $(I,\lambda)=\phi(R)$. Then:
\begin{itemize}
\item
each crossing of $I$ corresponds to a cell of $R$ containing a cross,
having a rook to its left (in the same row) and a rook below (in the same 
column).
\item
each triple $(i,k,j)$ such that $i<k<j$, $k\in$Fix$(I)$ and $(i,j)$ is an arch of $I$
corresponds to a cell of $R$ containing a cross, having either a rook to its left
(in the same row) or a rook below (in the same column).
\end{itemize}
\end{prop}

\begin{proof} These two statements are respectively illustrated  in the left part
and the right part of Figure \ref{mu}.
\begin{itemize}
\item Let $((i,j),(k,l))$ be a crossing of $I$. Since $k<j$, column $k$
intersects row $j$ in some cell $c$. Then, $(i,j)$ is an arch of $I$,
which means that there is a rook at the intersection of column $i$ and row 
$j$, to the left of the cell $c$. Similarly, $(k,l)$ is an arch so 
there is a rook at the intersection of column $k$ and row $l$, below the 
cell $c$. So to this crossing $((i,j),(k,l)$ we associate the cell $c$.

\item Let $(i,k,j)$ be such that $i<k<j$, $k\in$ Fix$(I)$ and $(i,j)$ is an 
arch of $I$.
We suppose for example that $k$ is the label of a column. Since $k<j$,
row $j$ intersects column $k$ in some cell $c$. There is no rook below
the cell $c$ because $k$ is a fixed point of $I$. But there is a rook in row 
$k$, to the left of $c$. So to this triple $(i,k,j)$ we associate the cell $c$.
\end{itemize}
\end{proof}

\begin{figure}[htp]\psset{unit=0.3cm}
\begin{pspicture}(1,-2)(9,2.5)
\psdots(1,0)(2,0)(3,0)(4,0)(5,0)(6,0)(7,0)(8,0)(9,0)
\psarc(3,0){2}{0}{180} \psarc(6.5,0){2.5}{0}{180}
\rput(1,-1){i}\rput(4,-1){k}\rput(5,-1){j}\rput(9,-1){l}
\end{pspicture}
\hspace{0.5cm}
\begin{pspicture}(0,-1)(5,4)
\psline[linecolor=gray](0.5,4)(0.5,2.5)(3,2.5)
\psline[linecolor=gray](2.5,3)(2.5,0.5)(5,0.5)
\psline(0,0)(5,0)  \psline(0,0)(0,4)
\psline(0,1)(5,1)  \psline(1,0)(1,4)
\psline(0,2)(5,2)  \psline(2,0)(2,4)
\psline(0,3)(3,3)  \psline(3,0)(3,3)
\psline(0,4)(2,4)  \psline(4,0)(4,2)
                   \psline(5,0)(5,2)
\rput(0.5,2.5){$\circ$}\rput(2.5,0.5){$\circ$}\rput(2.5,2.5){$\times$}
\rput(0.5,4.7){i}\rput(2.7,3.6){k}\rput(3.5,2.7){j}\rput(5.7,0.5){l}
\end{pspicture}
\hspace{2cm}
\begin{pspicture}(1,-2)(5,2.5)
\psdots(1,0)(2,0)(3,0)(4,0)(5,0)
\psarc(3,0){2}{0}{180}
\rput(1,-1){i}\rput(2,-1){k}\rput(5,-1){j}
\end{pspicture}
\hspace{0.5cm}
\begin{pspicture}(0,-2)(3,2)
\psline[linecolor=gray](0.5,2)(0.5,0.5)(3,0.5)
\psline[linecolor=gray](1.5,2)(1.5,0)
\psline(0,0)(3,0)  \psline(0,0)(0,2)
\psline(0,1)(3,1)  \psline(1,0)(1,2)
\psline(0,2)(2,2)  \psline(2,0)(2,2)
                   \psline(3,0)(3,1)
\rput(0.5,0.5){$\circ$}\rput(1.5,0.5){$\times$}
\rput(0.5,2.7){i}\rput(1.5,2.7){k}\rput(3.5,0.5){j}
\end{pspicture}
\caption{  Interpretation of crossings, and sum of heights of fixed points,
 in terms of rook placements. \label{mu}}
\end{figure}
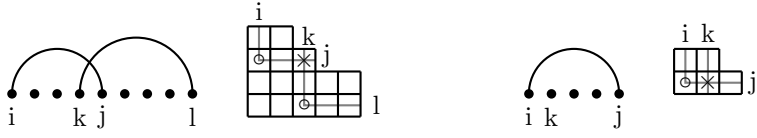

\begin{prop} If $\phi(R)=(I,\lambda)$ then the number of crosses in $R$ is 
$|\lambda| + \mu(I)$, where $\mu$ is the statistic on involutions
defined by
\[ \mu(I)=  cr(I) + \sum_{x\in \hbox{Fix}(I)} ht(x). \]
\end{prop}

\begin{proof} From the definition of the map $\beta$, we directly see that 
$|\lambda|=|\beta(R)|$ is the number of crosses in $R$ with no rook in the 
same row and no rook in the same column. Besides, from Proposition \ref{crosses}
we know that the number of crossings $cr(I)$ counts the crosses of $R$ with one 
rook to its left and one rook below. From the same proposition, we also know that 
the sum of heights of fixed points counts all remaining crosses.
\end{proof}

\smallskip

The previous proposition means that the number of crosses in rook placements
is an additive parameter with respect to the decomposition 
$R\mapsto(I,\lambda)$. This situation naturally leads to a factorisation
of the corresponding generating functions, so we get the following corollary:

\smallskip

\begin{corollary} For any j,k,n, we have
\begin{equation}
 T_{j,k,n} = \qbin{n-2k+2j}{j} T_{0,k-j,n}. \label{Tj0}
\end{equation}
\end{corollary}

\begin{proof} We assume that $n-2k+2j\geq 0$, since otherwise both sides are 0.
Indeed, a rook placement enumerated by $T_{j,k,n}$ contains $k-j$ rooks, so it has
at least $k-j$ different rows and $k-j$ different columns, so its half-perimeter
is at least $2k-2j$. Using the bijection $\phi$, we can compute $T_{j,k,n}$ by summing
the weights of couples $(I,\lambda)$ where $I\in$ Inv$(n,n-2k+2j)$ and $\lambda\in$
Par$(n-2k+j,j)$. Hence:

\[ T_{j,k,n} = p^{k-j} \sum_{(\lambda,I)} q^{|\lambda|+\mu(I)} 
= \Bigg( \sum_{\lambda} q^{|\lambda|} \Bigg) \Bigg(  p^{k-j}
\sum_{I} q^{\mu(I)}\Bigg). \]
The first factor of the right-hand side is $\qbin{n-2k+2j}{j}$ by a direct 
application of Proposition \ref{part}. The second factor can be seen as a sum 
over couples $(I,\lambda)$ where $\lambda$ has 0 rows and $n-2k+2j$ columns. 
So using again the bijection $\phi$, this second factor is $T_{0,k-j,n}.$
\end{proof}

\smallskip

Thanks to this factorisation property of $T_{j,k,n}$, the problem is reduced to
the evaluation of $T_{0,k,n}$. But this factorisation property also gives a 
recurrence relation satisfied by $T_{0,k,n}$.

\bigskip

\begin{corollary} We have the following recurrence relation:
\begin{equation}
  T_{0,k,n} = T_{0,k,n-1} + p [n+1-2k]_qT_{0,k-1,n-1}. \label{T0rec}
\end{equation}
\end{corollary}

\smallskip

\begin{proof} 
When $j=0$, the relation (\ref{Trec}) gives
\[ T_{0,k,n} = T_{0,k,n-1} + p  T_{1,k,n-1}. \]
Applying the previous corollary to the second term of this sum gives the desired 
equality.
\end{proof}

\bigskip

\section{Enumeration of rook placements}

\label{sec5}
\bigskip

In this section we solve the recurrence (\ref{T0rec}), and we obtain an 
expression for $T_{0,k,n}$ involving both $q$-binomials and Delannoy numbers, 
generalizing the two examples of Section \ref{sec3}. Using the factorisation property
of $T_{j,k,n}$ and summing over $j$, we obtain an expression for 
\[ T_{k,n} = \sum_{j=0}^k T_{j,k,n}, \]
{\it i.e.} for the sum of weights of rook placements of half-perimeter $n$
with $k$ rows. This expression is rather lengthy, with a sum over three 
indices, but for certain values of $p$ we can simplify it 
with the $q$-binomial identities of Lemma \ref{identities}. So in these 
particular specializations we get expressions for $T_{k,n}$ and 
$T_n$ without $q$-binomials.

\bigskip

\begin{prop} \label{T0knprop}
When $p=1-q$, we have
\begin{equation}
T_{0,k,n} (1-q,q) = \sum_{i=0}^k (-1)^i q^{\frac {i(i+1)} 2} 
  \qbin{n-2k+i}{i} \acc n {k-i}.  \label{T0kn}
\end{equation}
\end{prop}

\begin{proof} We give here a recursive proof. In the appendix we give an alternative
proof, which is much more combinatorial.

Let us denote by $f(k,n)$ the right-hand side of (\ref{T0kn}). The initial 
condition is $f(k,0) = T_{0,k,0}=\delta_{0k}$ so it remains to check relation
(\ref{T0rec}) when $p=1-q$. Let us define
\[ A = \qbin{n-1-2k+i}{i}, \qquad B = q^{n-2k}\qbin{n-1-2k+i}{i-1}, \]
\smallskip

\[ C = \acc{n-1}{k-i}, \qquad D= \acc{n-1}{k-i-1}, \]
so that we have
\[ f(k,n) = \sum_{i=0}^k (-1)^i q^{\frac {i(i+1)} 2} (A+B)(C+D) =
  \sum_{i=0}^k (-1)^i q^{\frac {i(i+1)} 2} \Big(AC + BC + (A+B)D \Big). 
\]

\noindent 
After expanding this sum, the second term gives
\[ \sum_{i=0}^k (-1)^i q^{\frac {i(i+1)} 2} BC  = - \sum_{i=0}^{k-1} (-1)^i
q^{\frac {(i+1)(i+2)} 2} q^{n-2k} \qbin{n-2k+i}{i}  \acc{n-1}{k-i-1}, \]
(the sum is reindexed such that $i$ becomes $i+1$). And the third term gives
\[ \sum_{i=0}^k (-1)^i q^{\frac {i(i+1)} 2} (A+B)D  =  \sum_{i=0}^{k-1} (-1)^i 
q^{\frac {i(i+1)} 2} \qbin{n-2k+i}{i}  \acc{n-1}{k-i-1}, \]
(after noticing that the term where $i=k$ is 0). Adding the previous two identities yields
\[ \sum_{i=0}^k (-1)^i q^{\frac {i(i+1)} 2} \big(BC+AD+BD\big) =  
\sum_{i=0}^{k-1} 
(-1)^i q^{\frac {i(i+1)} 2} \qbin{n-2k+i}{i} \acc{n-1}{k-i-1} 
\big( 1 - q^{n-2k+i+1}\big).\]
But we have $[n-2k+i+1]_q \qbin{n-2k+i}{i} = [n-2k+1]_q \qbin{n-2k+i+1}{i}$, 
hence
\[ \sum_{i=0}^k (-1)^i q^{\frac {i(i+1)} 2} \big(BC+AD+BD\big) =  \sum_{i=0}^{k-1} 
(-1)^i q^{\frac {i(i+1)} 2} \qbin{n-2k+i+1}{i}  \acc{n-1}{k-i-1} 
\big( 1 - q^{n-2k+1}\big)\]
\[ = \big( 1 - q^{n-2k+1}\big) f(k-1,n-1).\]
Since $\sum_{i=0}^k (-1)^i q^{\frac {i(i+1)} 2} AC$ readily gives $f(k,n-1)$, we
get the relation
\[ f(k,n) = f(k,n-1) + \big( 1 - q^{n-2k+1}\big) f(k-1,n-1), \]
which is precisely (\ref{T0rec}) when $p=1-q$.
\end{proof}

\bigskip

\noindent
{\it Remark:} The rook placements enumerated by $T_{0,k,n}$ contain exactly $k$ 
rooks, so $T_{0,k,n}(p,q) = p^kT_{0,k,n}(1,q)$. This shows that there is no loss of
generality in the assumption $p=1-q$ of the previous proposition. Moreover, the
Touchard-Riordan formula (\ref{tou_rio}) mentioned in the introduction is a
particular case of (\ref{T0kn}). Indeed, via the bijection of the previous section,
involutions without fixed points correspond to rook placements with exactly one
rook per row and one rook per column (therefore with as many rows as columns). So knowing
(\ref{T0kn}), we directly obtain (\ref{tou_rio}):
\[ \sum_{I \in \hbox{Inv}(2n,0)} q^{\hbox{cr}(I)} = T_{0,n,2n}(1,q) = 
\tfrac{1}{(1-q)^n}T_{0,n,2n}(1-q,q) = \tfrac{1}{(1-q)^n} \sum_{i=0}^n (-1)^i
  \acc{2n}{n-i} q^{\frac {i(i+1)}2}.
\]
\smallskip

Now using (\ref{Tj0}) and (\ref{T0kn}), we have the following equality:
\begin{equation}
T_{j,k,n}(1-q,q) = \qbin{n-2k+2j}{j}
\sum_{i=0}^{k-j} (-1)^i q^{\frac {i(i+1)} 2} \qbin{n-2k+2j+i}{i}
  \acc n {k-j-i}. \label{Tjkn}
\end{equation}
And as in the previous remark, $T_{j,k,n}(p,q) = p^{k-j} T_{j,k,n}(1,q)$ so that we
have a similar expression for any value of $p$. Summing it over $j$ will give an expression
for $T_{k,n}(p,q)$. For certain values of $p$, it is possible to simplify greatly 
this sum. To perform this simplification we need the following lemma:

\begin{lem} \label{identities}
For any $k,n\geq 0$ we have the following $q$-binomial identities:
\begin{equation}\label{id1}
\sum_{j=0}^k (-1)^j q^{\frac{j(j+1)}2} \qbin{n-j}{n-k} \qbin{n-k}{j} = 1,
\end{equation}
\begin{equation}\label{id2}
\sum_{j=0}^k (-1)^j q^{\frac{j(j-1)}2} \qbin{n-j}{n-k} \qbin{n-k}{j} = q^{k(n-k)},
\end{equation}
\begin{equation}\label{id3}
\sum_{j=0}^k (-1)^j q^{\frac{(j-1)(j-2)}2} \qbin{n-j}{n-k} \qbin{n-k}{j} = 
\frac{ q^{(k+1)(n-k)} - q^{k(n-k)} + q^{k(n+1-k)} - q^{(k+1)(n+1-k)} }
 {q^{n-1} (1-q) }.
\end{equation}
\end{lem}

\begin{proof} The first two are proved combinatorially using Proposition 
\ref{part}. We first prove (\ref{id2}), because it is slightly more 
simple. It seems that there is no simple combinatorial proof of (\ref{id3}) so 
we prove it with a recurrence, which is quite similar to the one of Proposition
\ref{T0knprop}.

\begin{itemize}
\item
The left-hand side of (\ref{id2}) counts the pairs 
$(\lambda,\mu)\in$ Par$(n-k,k-j)\times$Par$(n-k-1,j)$, for some $j$ between 0 and 
$k$, signed by $(-1)^j$ and such that $\mu$ has distinct parts. More 
precisely, $\lambda$ is such that 
$n-k\geq\lambda_1\geq\ldots\geq\lambda_{k-j}\geq0$ and $\mu$ is such that
$n-k > \mu_1>\ldots>\mu_j\geq0$. When $k-j>0$, such a couple $(\lambda,\mu)$ 
satisfying $\lambda_{k-j}<\mu_j$ or $\mu = (\emptyset)$
can be paired with the couple
$(\lambda',\mu')$ such that:
\[
\lambda'=(\lambda_1,\ldots,\lambda_{k-j-1}), \qquad
\mu'=(\mu_1,\ldots,\mu_j,\lambda_{k-j}).
\]
This couple satisfies $|\lambda|+|\mu|=|\lambda'|+|\mu'|$ but it has opposite sign.
The only couple which is not paired with any other is such that 
$\lambda_1=\ldots=\lambda_k=n-k$ and $\mu=(\emptyset)$, it contributes to the 
sum with a $q^{k(n-k)}$.

\item The proof of (\ref{id1}) is quite similar. Here the factor
$q^{j(j+1)/2}$ means that we count pairs $(\lambda,\mu)$ as before but such that
$n-k \geq \mu_1>\ldots>\mu_j>0$, (because $j(j+1)/2=1+\ldots+j$). Now
the pairing is done by comparing the smallest non-zero part of $\lambda$ with
the smallest part of $\mu$. Depending on the situation, one of these parts is 
moved from $\lambda$ to $\mu$, or from $\mu$ to $\lambda$.
The only couple $(\lambda,\mu)$ which is not paired with any other is such that
$\lambda_1=\ldots=\lambda_k=0$ and $\mu=(\emptyset)$, and it contributes to the 
sum with a 1.

\item When $k=0$, both sides of (\ref{id3}) are equal to $q$. Let us denote by $g(n,k)$ 
the left-hand side of (\ref{id3}). We define:
\[     
        A = q^{n-k}\tqbin{n-j}{n-k},\qquad B = \tqbin{n-j}{n-k-1}
\qquad  C = \tqbin{n-k-1}{j}, \qquad D = q^{n-k-j}\tqbin{n-k-1}{j-1},
\]
so that $g(n+1,k+1) = \sum_{j=0}^{k+1} (-1)^j q^{(j-1)(j-2)/2} (A+B)(C+D)$. After
expanding this product, we get the recurrence relation
\[ g(n+1,k+1) = q^{n-k}g(n,k) + g(n,k+1) - q^{n-k}g(n-1,k). \]
In view of the simple expression of the right-hand side of (\ref{id3}), it is
straightforward to check that it satisfies the same relation.
\end{itemize}
\end{proof}

\begin{prop} 
\begin{equation}
T_{k,n}(1-q,q)=\binom nk,\qquad T_{k,n}\left(\tfrac{1-q}q,q\right)=
\sum_{j=0}^k  \acc n j q^{(k-j)(n-k-j)},
\label{Tkn_0}
\end{equation}
\begin{equation}
T_{k,n}\left(\tfrac{1-q}{q^2},q\right) = \sum_{j=0}^k  \acc n j 
\left(  \tfrac{ q^{(k+1-j)(n-k-j)} - q^{(k-j)(n-k-j)} 
 + q^{(k-j)(n+1-k-j)} - q^{(k+1-j)(n+1-k-j)} }{(1-q)q^n} \right).
\label{Tkn}
\end{equation}
\end{prop}

\begin{proof} The three identities of this proposition respectively come
from (\ref{id1}), (\ref{id2}) and (\ref{id3}). We prove only the last one, 
because it is the most important case. The two others are proved similarly but
more simply. Multiplying (\ref{Tjkn}) by $q^{2j-2k}$ and summing over 
$j$ gives
\[
  T_{k,n} \left(\tfrac{1-q}{q^2},q\right) = \sum_{j=0}^k q^{2j-2k} \qbin{n-2k+2j}
  {j} \sum_{i=0}^{k-j} (-1)^i q^{\frac {i(i+1)} 2} \qbin{n-2k+2j+i}{i}
  \acc n{k-j-i} 
\]
\[
=  \sum_{\substack{0\leq i,j \\ i+j\leq k}} \acc n{k-j-i} q^{2j-2k} 
\qbin{n-2k+2j}{j} (-1)^i q^{\frac {i(i+1)} 2} \qbin{n-2k+2j+i}{i}.
\]
Introducing $l=k-j-i$, we get:
\[ 
  T_{k,n}\left(\tfrac{1-q}{q^2},q\right) = \sum_{l=0}^k  \acc n l 
  \sum_{j=0}^{k-l} q^{2j-2k} \qbin{n-2k+2j}{j} (-1)^{k-j-l} q^{\frac 
  {(k-j-l)(k-j-l+1)} 2} \qbin{n-k+j-l}{k-j-l},
\]
and after replacing $j$ with $k-l-j$ we also have:
\[ T_{k,n} \left(\tfrac{1-q}{q^2},q\right)= \sum_{l=0}^k  \acc n l 
  \sum_{j=0}^{k-l} q^{-2j-2l} \qbin{n-2l-2j}{k-l-j} (-1)^j 
  q^{\frac {j(j+1)} 2} \qbin{n-2l-j}{j}
\]
\[ 
= \sum_{l=0}^k  \acc n l \sum_{j=0}^{k-l} (-1)^j 
q^{\frac {(j-1)(j-2)}2 -1-2l} 
\frac{ [n-2l-j]_q! }{ [j]_q! [k-l-j]_q! [n-k-l-j]_q!  }
\] 
\[
= \sum_{l=0}^k  \acc n l q^{-1-2l}\sum_{j=0}^{k-l} (-1)^j 
q^{\frac {(j-1)(j-2)}2} \qbin{n-2l-j}{n-l-k} \qbin{n-l-k}{j}.
\] 
At this point we can apply (\ref{id2})
with $n'=n-2l$ and $k'=k-l$, and get (\ref{Tkn}).
\end{proof}

\noindent
{\it Remark:} By an obvious argument of symmetry by transposition, we have 
$T_{k,n} = T_{n-k,n}$, and this can be directly seen in (\ref{Tkn}). 
The summand $q^{(k+1-j)(n-k-j)} - q^{(k-j)(n-k-j)} 
 + q^{(k-j)(n+1-k-j)} - q^{(k+1-j)(n+1-k-j)}$ is unchanged when $j$ is 
replaced with $n+1-j$. Besides, we have $\acc n j = -\acc n {n+1-j}$, so 
\[ \sum_{j=k+1}^{n-k}  \acc n j \left(  \tfrac{ q^{(k+1-j)(n-k-j)} - q^{(k-j)(n-k-j)}
 + q^{(k-j)(n+1-k-j)} - q^{(k+1-j)(n+1-k-j)} }{(1-q)q^n} \right)  =0.    \]
A consequence is that in (\ref{Tkn}) instead of summing over $j$ between 0 and $k$, 
we can sum over $j$ between 0 and min$(k,n-k)$. This is also true for the second identity of
(\ref{Tkn_0}). In this form, it is clear that $T_{k,n} = T_{n-k,n}$.

\medskip

The last step of this section is the summing over $k$ to get an expression
for $T_n(\frac{1-q}{q^2}, q, y)$.

\medskip

\begin{prop} \label{propTn}
\begin{equation}  
(1-q)q^n T_n\left(\tfrac{1-q}{q^2}, q, y\right) = (1+y) G(n) - G(n+1), 
\label{Tn}
\end{equation}  
\[\hbox{where \quad } G(n)=\sum_{j=0}^{\lfloor \frac n2 \rfloor} 
 \acc n j \sum_{i=0}^{n-2j} y^{i+j-1} q^{i(n+1-2j-i)}.\]

\end{prop}

\begin{proof} First we define
\[ P_k = \sum_{i=0}^k y^i q^{i(k+1-i)}.\]
We have to multiply (\ref{Tkn}) by $y^k$, and sum over $k$ between 0 and $n$. This gives:
\[ \begin{split}
(1-q)q^n T_n \left(\tfrac{1-q}{q^2}, q, y\right)
= \sum_{0\leq j\leq k\leq n} y^k \acc n j 
\Big( q^{(k+1-j)(n-k-j)} - q^{(k-j)(n-k-j)} 
 + q^{(k-j)(n+1-k-j)} \\ -  q^{(k+1-j)(n+1-k-j)} \Big) 
\end{split} \]
\[ \begin{split}
= \sum_{j=0}^n \acc n j \Bigg( \sum_{k=j}^n y^k q^{(k+1-j)(n-k-j)} 
- \sum_{k=j}^n y^kq^{(k-j)(n-k-j)}  + \sum_{k=j}^n y^k q^{(k-j)(n+1-k-j)} \\ -
  \sum_{k=j}^n y^k q^{(k+1-j)(n+1-k-j)} \Bigg)  \end{split} \]
\[ 
\begin{split}
=\sum_{j=0}^n \acc n j \Bigg(
\sum_{i=1}^{n+1-j} y^{i+j-1}q^{i(n+1-2j-i)}
-\sum_{i=0}^{n-j} y^{i+j}q^{i(n-2j-i)}
+\sum_{i=0}^{n-j} y^{i+j}q^{i(n+1-2j-i)} \\ -
\sum_{i=1}^{n+1-j} y^{i+j-1}q^{i(n+2-2j-i)}
\Bigg),
\end{split}
\]

\noindent
after a reindexing of the second and third sums with $i=k-j$, 
and of the first and fourth sums with $i=k+1-j$. Since $(1-q)q^nT_n$ 
is a polynomial, we can discard all negative powers of $q$ appearing 
in these sums. Modulo non-positive powers of $q$, these four sums
are respectively equal to $y^{j-1}P_{n-2j}$, $y^jP_{n-1-2j}$, 
$y^jP_{n-2j}$, $y^{j-1}P_{n+1-2j}$. But we have to be careful when it
comes to the constant terms in $q$. These constant terms are 
respectively:
\[ 
[q^0]\sum_{i=1}^{n+1-j} y^{i+j-1}q^{i(n+1-2j-i)} = 
y^{n-j}\chi_{\{1 \leq n+1-2j\leq n-j+1\}},
\]
\[
[q^0]\sum_{i=0}^{n-j} y^{i+j}q^{i(n-2j-i)} = 
1+y^{n-j}\chi_{\{0 \leq n-2j \leq n-j\}},
\]
\[ 
[q^0]\sum_{i=0}^{n-j} y^{i+j}q^{i(n+1-2j-i)} = 
1+y^{n+1-j}\chi_{\{0 \leq n+1-2j\leq n-j\}} ,
\]
\[
[q^0]\sum_{i=1}^{n+1-j} y^{i+j-1}q^{i(n+2-2j-i)} =
y^{n+1-j}\chi_{\{1 \leq n+2-2j\leq n+1-j\}},
\]
where $\chi_P$ is either 0 or 1 whether the property $P$ is false or true.
We see that these constant terms in $q$ actually cancel two-by-two, so that it
remains:
\[
(1-q)q^nT_n\left(\tfrac{1-q}{q^2}, q, y\right) = \sum_{j=0}^n \acc n j 
\Big( (y^j+y^{j-1}) P_{n-2j}  - y^jP_{n-1-2j} - y^{j-1}P_{n+1-2j} \Big)
\]
\[
= (1+y)\sum_{j=0}^n \acc n j y^{j-1}P_{n-2j}
-\sum_{j=0}^{n+1}\left( \tacc n {j-1} + \tacc n j\right)y^{j-1} P_{n+1-2j}
\]
\[ =(1+y)G(n)-G(n+1), \qquad \hbox{where} \quad G(n)=\sum_{j=0}^n  \acc n j
 y^{j-1}P_{n-2j}. \]
Since the polynomial $P_{n-2j}$ is zero when $n-2j<0$, we can sum over $j$
between 0 and $\lfloor n/2\rfloor$ in the definition of $G(n)$, so that we get
(\ref{Tn}).
\end{proof}

\bigskip

\section{Application to permutation enumeration}

\label{sec6}
\bigskip

In the previous section we have computed $T_n$, which is also equal to
$\bra{W}(y\hat D+\hat E)^n\ket{V}$ thanks to the results of Section \ref{sec2}. Now, using
the inversion formula (\ref{inv1}), we can compute $\bra{W}(yD+E)^n\ket{V}$ and
prove Theorem \ref{main}. At the beginning of this section we describe the 
combinatorial interpretation of this polynomial in terms of permutations and
permutation tableaux. Then we prove Theorem \ref{main} and Theorem
\ref{main2}, and give some applications.

\begin{prop}  \cite{Co,CN,CW,KSZ,SW,XGV2} \label{comb}
For any $n\geq 1$ the following polynomials are equal:
\begin{itemize}
\item $\bra{W}y(yD+E)^{n-1}\ket{V}$,
\item the generating function for permutation tableaux of size $n$, the number 
of lines counted by $y$ and the number of superfluous 1's counted by $q$,
\item the generating function for permutations of size $n$, the number of ascents
plus 1 counted by $y$ and the occurrences of the pattern 13-2 counted by $q$, 
\item the generating function for permutations of size $n$, the number of weak
excedances counted by $y$ and the number of crossings counted by $q$,
\item the $n$th moment of the $q$-Laguerre polynomials.
\end{itemize}
\end{prop}

\begin{proof} All this material is present in the references. See also the 
references for definitions. In particular there are several possible definitions
for the $q$-Laguerre polynomials: the one we mention is defined as a rescaled version 
of the Al-Salam-Chihara polynomials as in \cite{KSZ}. We recall that the $n$th moment 
of these $q$-Laguerre polynomials is the sum of weights of {\it histoires de Laguerre}
of $n$ steps. This is also present in \cite{Co}.

\begin{definition} \label{histoires} An histoire de Laguerre is a weighted Motzkin
path such that:
\begin{itemize}
\item the weight of an horizontal step at height $h$ is $q^i$ for some
 $i\in\{0,\dots,h-1\}$ or $yq^i$ for some $i\in\{0,\dots,h\} $,
\item the weight of a North-East step starting at height $h$ 
is $q^i$ for some $i\in\{0,\dots,h\}$,
\item the weight of a South-East step starting at height $h$ is $yq^i$ for some
$i\in\{0,\dots,h-1\}$.
\end{itemize}
\end{definition}

The classical bijections between permutations and histoires de Laguerre, namely
the Fran\c{c}on-Viennot and Foata-Zeilberger bijections, give the equality of the
last three items in the list of Proposition \ref{comb}.

\smallskip

As said in the introduction, the link between the operators $D$ and $E$ of the 
matrix Ansatz and the permutation tableaux was first exposed by Corteel and 
Williams in \cite{CW}. This shows the equality of the first two items in the list.
See also \cite{XGV2}.

\smallskip

To end this proof we can use the bijection between permutation
tableaux and permutations exposed in \cite{CN}: the number of columns in permutation
tableaux corresponds to the number of ascents in permutations, and the number
of superfluous 1's corresponds to the number of occurrences of the pattern 13-2. We also 
have to mention the previous results of Postnikov, who made the link between
\hbox{\rotatedown{$\Gamma$}}-diagrams, which generalize the permutation
tableaux, and alignments in decorated permutations \cite{AP,LW}.
\end{proof}

We now give the formula for the polynomials of Proposition \ref{comb}. This
is the Theorem \ref{main} stated in the introduction.

\begin{theo} For any $n\geq 1$, we have
\[
\bra{W} (yD+E)^{n-1} \ket{V} = \tfrac 1{y(1-q)^n} \sum_{k=0}^n (-1)^k 
\left(\sum_{j=0}^{n-k} y^j\Big( \tbinom{n}{j}\tbinom{n}{j+k} - 
   \tbinom{n}{j-1}\tbinom{n}{j+k+1}\Big)\right)
\left(\sum_{i=0}^k y^iq^{i(k+1-i)}  \right).
\]
\end{theo}

\begin{proof} Using the main result of the previous section (\ref{Tn})
and the inversion formula (\ref{inv1}), we obtain: 

\[
 \bra{W}(1-q)^n(yD+E)^{n-1} \ket{V} = (1-q)\sum_{k=0}^{n-1} \tbinom {n-1} k 
(1+y)^{n-1-k} (-q)^k\bra{W}(y\hat D+\hat E)^k\ket{V}
\]

\[
= \sum_{k=0}^{n-1} \tbinom {n-1} k (1+y)^{n-1-k}
(-1)^k \Big( (1+y)G(k) - G(k+1)  \Big) 
= \sum_{k=0}^n \tbinom nk (1+y)^{n-k} (-1)^k G(k) 
\]

\[
= \sum_{\substack{0\leq i \leq k \leq n\\ i\equiv k \hbox{ mod }2 }} 
\tbinom n k (1+y)^{n-k} (-1)^k \tacc{k}{\frac {k-i}2 }y^{(k-i)/2-1} P_i
= \frac 1y \sum_{i=0}^n (-1)^i \left(\sum_{k=0}^{\lfloor \frac{n-i}{2}\rfloor}
\tbinom{n}{2k+i}(1+y)^{n-2k-i}\tacc{2k+i}{k} y^k \right) P_i,
\]
after a reindexing such that $k$ becomes $2k+i$.
It remains to simplify the sum between parentheses. After expanding
the power of $1+y$, this sum is:

\[
\sum_{k=0}^{\lfloor \frac{n-i}{2}\rfloor} \sum_{j=0}^{n-2k-i}
\binom{n}{2k+i}\binom{n-2k-i}{j} \acc{2k+i}{k} y^{k+j}
\]

\[
= \sum_{0\leq k,j} \frac{n!}{j!(n-2k-i-j)!}
\left( \frac{1}{k!(k+i)!} - \frac{1}{(k-1)!(k+i+1)!} \right) y^{k+j} 
\]

\[
= \sum_{0\leq k \leq m} \frac{n!}{(m-k)!(n-m-k-i)!}
\left( \frac{1}{k!(k+i)!} - \frac{1}{(k-1)!(k+i+1)!} \right)  y^m
\]

\[
= \sum_{m=0}^{n-i} y^m \left( \binom{n}{m}\sum_{k=0}^m\tbinom{m}{k}
\tbinom{n-m}{k+i}-\binom{n}{m-1}\sum_{k=0}^m \tbinom{m-1}{k-1}
\tbinom{n-m+1}{k+i+1}\right).
\]
But thanks to the Vandermonde identity, the two sums over $k$ may be simplified:
\[
\sum_{k=0}^m \tbinom{m}{k}\tbinom{n-m}{k+i} =\binom{n}{m+i},
 \qquad
\sum_{k=0}^m \tbinom{m-1}{k-1}\tbinom{n-m+1}{k+i+1}=
\binom{n}{m+i+1},
\]
and this completes the proof.
\end{proof}

\medskip

\noindent
{\it Remark:} The number $\tbinom{n}{j}\tbinom{n}{j+k} - 
\tbinom{n}{j-1}\tbinom{n}{j+k+1}$ may be seen as the determinant of a 
$2\times 2$-matrix of binomial coefficients. The Lindst\"om-Gessel-Viennot lemma gives 
a combinatorial interpretation of this quantity in terms of lattice paths: 
it is the number of pairs of non-intersecting paths with starting points $(1,0)$ 
and $(0,1)$, with end points $(j,n-j+1)$ and $(j+k+1, n-k-j)$, and only with
unit steps going North or East, as in Figure \ref{paths}. In particular when $k=0$, 
this is the Narayana number $N(n+1,j+1 )$.

\begin{figure}[htp] \psset{unit=3mm}
\begin{pspicture}(-1,-1)(7,7)
\psline{->}(0,0)(7,0) \psline{->}(0,0)(0,7)
\psline(0,0)(7,0)  \psline(0,0)(0,7)
\psline(0,1)(7,1)  \psline(1,0)(1,7)
\psline(0,2)(7,2)  \psline(2,0)(2,7)
\psline(0,3)(7,3)  \psline(3,0)(3,7)
\psline(0,4)(7,4)  \psline(4,0)(4,7)
\psline(0,5)(7,5)  \psline(5,0)(5,7)
\psline(0,6)(7,6)  \psline(6,0)(6,7)
\psdots(1,0)(0,1)(3,6)(6,3)
\psline[linewidth=0.8mm](1,0)(1,1)(3,1)(3,2)(5,2)(5,3)(6,3)
\psline[linewidth=0.8mm](0,1)(0,2)(2,2)(2,4)(3,4)(3,6)
\rput(-1,-1){\small 0}\rput(1,-1){\small 1}\rput(3,-1){\small$j$}\rput(8,-1){\small$j+k+1$}
\rput(-1,1){\small 1}\rput(-3,3){\small$n-k-j$}\rput(-3,6){\small$n-j+1$}
\end{pspicture}
\caption{\label{paths} Interpretation of a determinant of binomials in terms
of lattice paths. In this example, we have $n=8$, $j=3$, $k=2$.}
\end{figure}
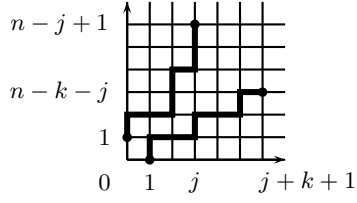

\begin{prop} \label{coef}
The coefficient of $y^m$ in $\bra{W}(yD+E)^{n-1}\ket{V}$ is given by:
\[
[y^m] \bra{W}(yD+E)^{n-1}\ket{V} = \frac 1{(1-q)^n}
\sum_{k=0}^n\sum_{j=m-k}^m (-1)^k q^{(m-j)(k+j+1-m)}
\Big( \tbinom{n}{j}\tbinom{n}{j+k} - \tbinom{n}{j-1}\tbinom{n}{j+k+1}\Big).
\]
\end{prop}

\begin{proof} We just have to expand the products in the equality of Theorem 
\ref{main}, since each of the factors between parentheses is a polynomial in 
$y$ and their coefficients are explicit.
\end{proof}

\medskip

In \cite{LW}, Williams provides a different formula for the same polynomial,
indeed $[y^k] \bra{W}(yD+E)^{n-1}\ket{V}$ is also equal to:
\[ \sum_{i=0}^{k-1}(-1)^i[k-i]^nq^{ki-k^2}
\left( \binom n i q^{k-i} + \binom n{i-1}
\right).
\]
She shows that this polynomial is a $q$-analog of Eulerian
numbers that interpolates between Narayana number (when $q=0$), binomial
coefficents (when $q=-1$), and of course Eulerian numbers (when $q=1$).

\bigskip

We can also obtain these results from the expression of Proposition \ref{coef}.
For example, if we put $q=0$ in the previous equality, it tells that the
number of permutations avoiding the pattern 13-2 and with $m$ ascents is:
\[ \sum_{k=0}^n(-1)^k\left( \tbinom{n}{m}\tbinom{n}{m+k} - 
\tbinom{n}{m-1}\tbinom{n}{m+k+1}\right) = 
\binom n m ^2 + \sum_{k=1}^n (-1)^k \tbinom{n}{m}\tbinom{n}{m+k}
 + \sum_{k=1}^{n+1} (-1)^k \tbinom{n}{m-1}\tbinom{n}{m+k}
\]

\[ 
= \binom n m ^2 +�\sum_{k=1}^n (-1)^k \binom{n+1}m \binom n{m+k}
= \binom n m ^2 - \binom {n+1}m \sum_{k=0}^m (-1)^{k+m}\binom n k.
\]
This alternating sum of binomials is also the binomial $\tbinom{n-1}m$.
So the number we get is $\tbinom n m ^2 - \tbinom {n+1}m \tbinom{n-1}m$.
Although it is not the most common way to define it, this number is
the Narayana number $N(n,m)$, as can be combinatorially seen  using again
the Lindstr\"om-Gessel-Viennot lemma.

\bigskip

We now give the specialization when $y=1$. This is the Theorem \ref{main2}
stated in the introduction.

\medskip

\begin{theo} For any $n\geq 1$, we have
\begin{equation}
\bra{W}(D+E)^{n-1}\ket{V} = \frac 1{(1-q)^n} \sum_{k=0}^n (-1)^k 
\left( \binom{2n}{n-k} - \binom{2n}{n-k-2} \right)
\Bigg( \sum_{i=0}^k q^{i(k+1-i)} \Bigg). 
\end{equation}
\end{theo}

\begin{proof} We just have to substitute $y=1$ into the equality of Theorem \ref{main}.
We can simplify the resulting expression using again the Vandermonde identity,
indeed we have

\[
\sum_{j=0}^{n-k} \tbinom{n}{j}\tbinom{n}{j+k} = \sum_{j=0}^{n-k} \tbinom{n}{j}
\tbinom{n}{n-k-j} = \binom{2n}{n-k},
\]

\[
\sum_{j=0}^{n-k} \tbinom{n}{j-1}\tbinom{n}{j+k+1} = \sum_{j=0}^{n-k} 
\tbinom{n}{j-1} \tbinom{n}{n-j-k-1} = \binom{2n}{n-k-2},
\] 
and the result follows.
\end{proof}

\bigskip

Among the several objects of the list in Proposition \ref{comb}, the most 
studied are probably permutations and the pattern 13-2, see for example
\cite{CM,CN,SW,RP}. In particular in \cite{CM,RP} we can find methods for
obtaining, as a function of $n$ for a given $k$, the number of permutations 
of size $n$ with exactly $k$ occurrences of the pattern 13-2. By taking the Taylor 
series of (\ref{main2}), we obtain direct and quick proofs for these previous 
results. As an illustration we give the formulas for $k\leq 3$ in the following
proposition.

\begin{prop} The order 3 Taylor series of $\bra{W}(D+E)^{n-1}\ket{V}$ is
\[
\bra{W}(D+E)^{n-1}\ket{V}=C_n + \binom{2n}{n-3}q + \frac n2\binom{2n}{n-4}q^2
+ \frac{(n+1)(n+2)}{6}\binom{2n}{n-5}q^3+O(q^4),
\]
where $C_n$ is the $n$th Catalan number.
\end{prop}

\begin{proof} On one side, we have $(1-q)^{-n}=1+nq+\tbinom{n+1}2 q^2 + 
\tbinom{n+2}3 q^3+O(q^4)$. On the other side, we have $\sum_{i=0}^k 
q^{i(k+1-i)}=1+q\delta_{1k}+2q^2\delta_{2k}+2q^3\delta_{3k}+O(q^4)$.
The constant term is:
\[ \sum_{k=0}^n \left(\binom{2n}{n-k}-\binom{2n}{n-k-2}\right) = 
\binom{2n}{n}-\binom{2n}{n-1}=C_n. \]
So this Taylor series is

\[ \left(1+nq + \tbinom{n+1}2 q^2 + \tbinom{n+2}3 q^3\right)
\left( C_n - \left(\tbinom{2n}{n-1}-\tbinom{2n}{n-3}\right) q 
 + \left(\tbinom{2n}{n-2}-\tbinom{2n}{n-4}\right) q^2
 - \left(\tbinom{2n}{n-3}-\tbinom{2n}{n-5}\right) q^3
\right).
\]
After expanding the product, all coefficients can be seen as the product of
$\tbinom{2n}{n}$ and a rational fraction of $n$. So the simplification is just 
a matter of simplifying rational fractions of $n$, which is straightforward.
\end{proof}

\medskip

More generally, a computer algebra system can provide higher order terms,
for example it takes no more than a few seconds to obtain the following 
closed formula for $[q^{10}]\bra{W}(D+E)^{n-1}\ket{V}$:
\[
\tfrac {(2n)!}{10!(n+12)!(n-8)!} \Big( {n}^{13}+70\,{n}^{12}+2093\,{n}^{11}+
32354\,{n}^{10}+228543\,{n}^{9}-318990\,{n}^{8}-17493961\,{n}^{7}
 - 104051458\,{n}^{6} \]
\[ -6828164\,{n}^{5} 
+2022876520\,{n}^{4}+6310831968\,{n
}^{3}+5832578304\,{n}^{2}+14397419520\,n+5748019200 \Big),
\]
which is quite an improvement when compared to the
methods of \cite{RP}. Besides these exact formulas, the following proposition
gives the asymptotic for permutations with a given fixed number of occurrences
of the pattern 13-2.

\medskip

\begin{theorem} \label{asymp}
for any $m\geq 0$ we have the following asymptotic when $n$ goes to infinity:
\[ [q^m]\bra{W}(D+E)^{n-1}\ket{V}  \sim 
\frac{4^nn^{m-\frac 32}}{\sqrt{\pi} m!}. \]
\end{theorem}

\begin{proof} When $n$ goes to infinity, the numbers $\tbinom{2n}{n-k} - 
\tbinom{2n}{n-k-2}$ are dominated by the Catalan number $\frac 1{n+1} 
\tbinom{2n}{n}$. It implies that in $(1-q)^n\bra{W}(D+E)^{n-1}\ket{V}$, each higher 
order term grows at most as fast as the constant term $C_n$. On the other side,
the coefficient of $q^m$ in $(1-q)^{-n}$ is equivalent to $n^m / m!$. So we
have the asymptotic
\[ [q^m]\bra{W}(D+E)^{n-1}\ket{V} \sim 
\frac{C_nn^m}{m!}. \]
Knowing the asymptotic of the Catalan numbers, we can conclude the proof.
\end{proof}

\smallskip

Since any occurrence of the pattern 13-2 in a permutation is also an occurrence 
of the pattern 1-3-2, a permutation with $k$ occurrences of the pattern 1-3-2 
has at most $k$ occurrences of the pattern 13-2. So we get the following 
corollary.

\medskip

\begin{corollary} Let $\psi_k(n)$ be the number of permutations in $\mathfrak{S}_n$
with at most $k$ occurrences of the pattern 1-3-2.
For any constant $C>1$ and $k\geq 0$, we have
\[ \psi_k(n) \leq C\frac{4^nn^{k-\frac 32}}{\sqrt{\pi} k!}\]
when $n$ is sufficiently large.
\end{corollary}

\begin{proof} By the previous remark we have
\[ \psi_k(n) \leq \sum_{i=0}^k [q^i]\bra{W}(D+E)^{n-1}\ket{V}, \]
so this is a consequence of Theorem \ref{asymp}, which gives the asymptotics of 
each of these terms.
\end{proof}

\medskip

So far we have mainly used Theorem \ref{main2}. Now we illustrate what 
we can do with the refined formula given in Theorem \ref{main}.
We already mentioned that we get Narayana numbers when $q=0$, but
we can also get the coefficients of higher degree in $q$. 
For example it is conjectured in \cite{LW} that the coefficient of 
$qy^m$ in $\bra{W}y(yD+E)^{n-1}\ket{V}$ is equal to
$\tbinom{n}{m+1}\tbinom{n}{m-2}$. With our results we can prove:

\begin{prop} The coefficients of $qy^m$ and $q^2y^m$ in $\bra{W}y(yD+E)^{n-1}\ket{V}$ 
are respectively
\[ \binom{n}{m+1}\binom{n}{m-2} \qquad \hbox{ and } \qquad \binom{n+1}{m-2}
\binom{n+1}{m+2}\frac{nm+m-m^2-4}{2(n+1)}. \]
\end{prop}

\begin{proof} A naive expansion  of the Taylor series in $q$ gives a lengthy formula, which
is simplified straightforwardly after noticing it is the product of $\tbinom n m ^2$ and a 
rational fraction of $n$ and $m$.
\end{proof}

\bigskip

\section*{Appendix}

We give here a combinatorial proof of Proposition \ref{T0knprop}. As noticed earlier,
this result is a generalization of the Touchard-Riordan formula (\ref{tou_rio}), and 
this combinatorial proof is a generalization of Penaud's combinatorial proof \cite{JGP}
of (\ref{tou_rio}). We follow very closely this reference, even in some notations. Moreover
the ideas of this proof were inspired by the alternative proof of Theorem \ref{main}
mentioned in the introduction (see \cite{CJPR}).

\begin{prop} There is a bijection between involutions
on $\{1,\dots,n\}$ and weighted Motzkin paths of $n$ steps with the 
following properties:
\begin{itemize}
\item The weight of an East step at height $h$ is $q^h$.

\item The weight of a Sout-East step starting at height $h$ is $q^i$
for some $i\in\{0,\dots,h-1\}$.
\end{itemize}
Moreover the image of an involution $I$ on $\{1,\dots,n\}$ is a weighted Motzkin 
path with total weight $q^{\mu(I)}$.
\end{prop}

\begin{proof}
This is obtained via the same methods as the bijection between 
involution without fixed points and Hermite histories, see \cite{JGP}.
It is also very similar to the Foata-Zeilberger bijection.
See Figure \ref{hist} for an example.
\end{proof}

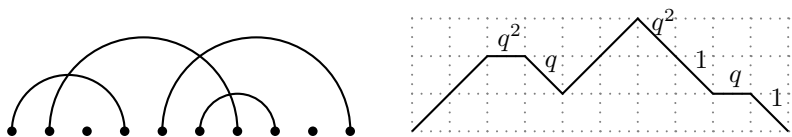
\begin{figure}[h!tp]\psset{unit=5mm}
\begin{pspicture}(1,0)(10,2.5)
\psdots(1,0)(2,0)(3,0)(4,0)(5,0)(6,0)(7,0)(8,0)(9,0)(10,0)
\psarc(2.5,0){1.5}{0}{180}\psarc(4.5,0){2.5}{0}{180}
\psarc(7.5,0){2.5}{0}{180}\psarc(7,0){1}{0}{180}
\end{pspicture}
\qquad
\begin{pspicture}(0,0)(10,3)
\psgrid[gridcolor=gray,griddots=4,subgriddiv=0,gridlabels=0](0,0)(10,3)
\psline(0,0)(1,1)(2,2)(3,2)(4,1)(5,2)(6,3)(7,2)(8,1)(9,1)(10,0)
\rput(2.6,2.5){$q^2$}\rput(3.7,1.8){$q$}\rput(6.7,2.9){$q^2$}
\rput(7.7,1.9){$1$}\rput(8.6,1.4){$q$}\rput(9.7,0.9){$1$}
\end{pspicture}
\caption{An involution and the corresponding weighted Motzkin path. \label{hist}}
\end{figure}

To compute $T_{0,k,n}(1,q)$, we have to sum the weights of the 
weighted Motzkin paths having $n$ steps, and $n-2k$ East steps. 
When we multiply by $(1-q)^k$, there are many cancellations in this sum.
Indeed we easily see that to compute $T_{0,k,n}(1-q,q)$, we have to sum the 
weights of Motzkin paths of $n$ steps satisfying conditions $(C2)$:
\begin{itemize}
\item the weight of an East step at height $h$ is $q^h$.
\item the weight of a Sout-East step starting at height $h$ is either 1 or $-q^h$.
\end{itemize}
Now, we give a decomposition of these weighted Motzkin paths.

\begin{prop} \label{decomp}
There is a weight-preserving bijection between weighted Motzkin paths satisfying
$(C2)$ and couples $(H_1,H_2)$ such that for some $i\in\{0,\dots,k\}$,
\begin{itemize}
\item $H_1$ is a left factor of a Dyck path, with $n$ steps and ending
  at height $n-2k+2i$,
\item $H_2$ is a weighted Motzkin path of $n-2k+2i$ steps, 
with $n-2k$ East steps, statifying conditions $(C2)$ above, and also that
any South-East step following a North-East step has weight $-q^h$ ({\it i.e.} not 1).
\end{itemize}
\end{prop}

\begin{proof} This is similar to Lemma 1 in \cite{CJPR}.
\end{proof}

A weighted Motzkin path as $H_2$ above is called a {\it core}.
The enumeration of left factors of Dyck path is given by Delannoy numbers.
On the other hand, to compute the sum of weights of cores we need two other lemmas.

\begin{lem} \label{invol}
There is an involution $\gamma_i$ on cores of length $n-2k+2i$ with $n-2k$ East steps, 
with the following properties:
\begin{itemize}
\item if a core and its image are different they have opposite weights,
\item the fixed points of $\gamma_i$ are the cores such that:
\begin{itemize}
\item the $i$ first steps are North-East, and all following steps are East
or South-East,
\item a South-East step starting at height $h$ has weight $-q^h$ ({\it i.e.} not 1).
\end{itemize}
\end{itemize}
\end{lem}

\begin{proof} In this proof we use a word notation for cores: the letters $x$, 
$z$, $y$, and $\bar{y}$ respectively correspond to North-East steps, East steps, 
South-East steps weighted by 1, and South-East steps weighted by $-q^h$.
For a core $c$, let $u(c)$ be the length the last sequence of consecutive $x$'s.
Let $v(c)$ be the height of the last $y$ if there is no $x$ after this $y$, and $i$ otherwise. The fixed points of $\gamma_i$ are the cores such that $u(c)=v(c)=i$. 

\smallskip

From now on we assume
that $c$ does not satisfy $u(c)=v(c)=i$.
The involution $\gamma_i$ is such that $u(c)\geq v(c)$ if and only if $u(\gamma_i(c)) < 
v(\gamma_i(c))$.
Suppose that $u(c)\geq v(c)$. Let $\tilde{c}$ be the word obtained from $c$ when we
replace the last $y$ with a $\bar{y}$.
There is a unique factorization $\tilde{c}=f_1x^{u(c)} ay^jf_2$ such that:
\begin{itemize}
\item $a$ is either $z$ or $\bar{y}$,
\item $f_2$ begins with a $\bar{y}$ and contains no $x$.
\end{itemize}
We set
\[  \gamma_i(c) =  f_1x^{u(c)-v(c)}ay^j x^{v(c)}f_2. \]
See Figure \ref{invol_cores} for an example.

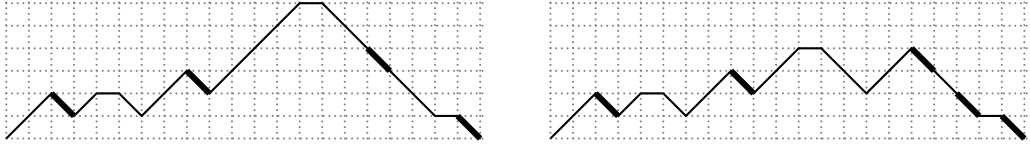
\begin{figure}[h!tp]\psset{unit=3mm}
\begin{pspicture}(0,0)(20,6)
\psgrid[gridcolor=gray,griddots=4,subgriddiv=0,gridlabels=0](0,0)(21,6)
\psline(0,0)(1,1)(2,2)(3,1)(4,2)(5,2)(6,1)(7,2)(8,3)(9,2)(10,3)(11,4)(12,5)(13,6)(14,6)(15,5)(16,4)(17,3)(18,2)(19,1)(20,1)(21,0)
\psline[linewidth=0.8mm](2,2)(3,1)\psline[linewidth=0.8mm](8,3)(9,2)
\psline[linewidth=0.8mm](16,4)(17,3)\psline[linewidth=0.8mm](20,1)(21,0)
\end{pspicture}
\hspace{1cm}
\begin{pspicture}(0,0)(20,6)
\psgrid[gridcolor=gray,griddots=4,subgriddiv=0,gridlabels=0](0,0)(21,6)
\psline(0,0)(1,1)(2,2)(3,1)(4,2)(5,2)(6,1)(7,2)(8,3)(9,2)(10,3)(11,4)(12,4)(13,3)(14,2)(15,3)(16,4)(17,3)(18,2)(19,1)(20,1)(21,0)
\psline[linewidth=0.8mm](2,2)(3,1)\psline[linewidth=0.8mm](8,3)(9,2)
\psline[linewidth=0.8mm](16,4)(17,3)\psline[linewidth=0.8mm](18,2)(19,1)
\psline[linewidth=0.8mm](20,1)(21,0)
\end{pspicture}
\caption{ A core $c$ and its image by $\gamma_i$. The thick lines indicate the 
$\bar{y}$, {\it i.e.} the
South-East steps weighted by $-q^h$. In this example we have $n-2k=3$,
$i=9$, $u=4$, $v=2$. We can check that $w(c)=-q^{17}= - w(\gamma_i(c))$.
\label{invol_cores}}
\end{figure}

\noindent
Simple arguments of word combinatorics show that:
\begin{itemize}
\item $c$ and its image have opposite weights,
\item any core $c'$ such that $u(c')<v(c')$ is obtained as a $\gamma_i(c)$ for some $c$
satisfying $u(c)\geq v(c)$. Indeed, let $\tilde{c}'$ be the word obtained from $c'$ by 
replacing the last $\bar{y}$ at height $u(c')$ with a $y$. There is unique factorization
$\tilde{c}'=f_1ay^jx^{u(c')}f_2$, where $a$ is $z$ or $\bar{y}$ and $f_2$ contains no $x$.
Then $c=f_1x^{u(c')}ay^jf_2$ has the required properties.
\end{itemize}
These arguments, put together, show that $\gamma_i$ has the claimed properties.
\end{proof}

\begin{lem} \label{cores}
The sum of weights of the fixed points of $\gamma_i$ is
\[ \sum_{H_2\in Fix(\gamma_i)} w(H_2) = (-1)^i q^{\frac{i(i+1)}2}\qbin{n-2k+i}{i}. \]
\end{lem}

\begin{proof}
A fixed point of $\gamma_i$ is fully characterized by the heights 
$h_1,\dots,h_{n-2k}$ of the $n-2k$ East steps, and these heights
can take any values such that $i\geq h_1 \geq \dots \geq h_{n-2k} 
\geq 0$. Such a fixed point of $\gamma_i$ has weight
\[  (-1)^iq^{\frac{i(i+1)}2} q^{\sum h_i}, \]
indeed the South-East steps have weights $-q^i,\dots,-q^2,-q$ and they
correspond to the factor $ (-1)^iq^{\frac{i(i+1)}2} $.
It remains to sum over $h_i$ and we can conclude thanks to Proposition 
\ref{part}.
\end{proof}

Now we can prove:

\setcounter{prop}{11}

\begin{prop}
\[
T_{0,k,n} (1-q,q) = \sum_{i=0}^k (-1)^i q^{\frac {i(i+1)} 2} 
  \qbin{n-2k+i}{i} \acc{n}{k-i}.
\]
\end{prop}

\begin{proof} The decomposition of weighted Motzkin paths stated in Proposition
\ref{decomp} gives
\[ T_{0,k,n}(1-q,q) = \sum_{i=0}^k \acc{n}{k-i} \sum_{H_2} w(H_2),  \]
where the second sum is over cores $H_2$ of $n-2k+2i$ steps with $n-2k$ East steps.
Thanks to Lemma \ref{invol}, we can restrict the second sum to the fixed points of
the involution $\gamma_i$. And thanks to Lemma \ref{cores}, this sum is
\[ \sum_{H_2} w(H_2) = (-1)^i q^{\frac{i(i+1)}2}\qbin{n-2k+i}{i}. \]
This completes the proof.
\end{proof}



\begin{thebibliography}{999}

\bibitem{Bu}
A. Burstein,  On some properties of permutation tableaux,
Ann. Combin. 11(3-4), (2007), 355-368

\bibitem{BECE}
R. A. Blythe, M. R. Evans, F. Colaiori and F. H. L. Essle,
Exact solution of a partially asymmetric exclusion model using a deformed 
oscillator algebra, J. Phys. A: Math. Gen. Vol. 33, (2000), 2313-2332.

\bibitem{CM}
A. Claesson, T. Mansour,
Counting Occurrences of a Pattern of Type (1,2) or (2,1) in Permutations,
Adv. in App. Maths. 29, (2002), 293-310.

\bibitem{Co}
S. Corteel, Crossings and alignments of permutations, Adv. in App. Maths.
38(2), (2007), 149-163.

\bibitem{CJPR}
S. Corteel, M. Josuat-Verg\`es, T. Prellberg, R. Rubey,
Matrix Ansatz, lattice paths and rook placements, submitted to FPSAC '09.

\bibitem{CN}
S. Corteel and P. Nadeau, Bijections for permutation tableaux,
Eur. Jal. of Comb 30(1), (2009), 295-310.

\bibitem{CW}
S. Corteel and L. K. Williams, Tableaux combinatorics for the asymmetric 
exclusion process, Adv. in App. Maths. 39(3), (2007), 293-310.

\bibitem{DEHP}
B. Derrida, M. Evans, V. Hakim, V. Pasquier, Exact solution of a 1D asymmetric 
exclusion model using a matrix formulation, J. Phys. A: Math. Gen. 26 (1993), 
1493-1517.

\bibitem{GR}
A. Garsia and J. Remmel, q-Counting rook configurations and a formula of
Frobenius. J. Combin. Theory, Ser. A 41, (1986), 246-275.

\bibitem{KSZ}
A. Kasraoui, D. Stanton, J. Zeng, The combinatorics of Al-Salam-Chihara 
$q$-Laguerre polynomials, Preprint 2008, http://arXiv.org/abs/0810.3232v1.

\bibitem{Ke}
S. Kerov, Rooks on Ferrers Boards and Matrix Integrals, Zapiski. Nauchn. 
Semin. POMI, v.240 (1997), 136-146.

\bibitem{AP}
A. Postnikov, Total positivity, Grassmannians, and networks, Preprint 2006.

\bibitem{RPS}
R. P. Stanley, Enumerative combinatorics Vol. 1, Cambridge university press 
(1986).

\bibitem{SW} E. Steingr\'imsson, L. K. Williams, Permutation tableaux and
permutation patterns, J. Combin. Theory Ser. A, Vol. 114(2),
(2007), 211-234.

\bibitem{RP}
R. Parviainen, Lattice path enumeration of permutations with $k$ occurrences of the 
pattern 2-13, Journal of Integer Sequences, Vol. 9 (2006), Article 06.3.2.

\bibitem{JGP} J.-G. Penaud, A bijective proof of a Touchard-Riordan formula,
Disc. Math., Vol. 139 (1995), 347-360.

\bibitem{TP}
T. Prellberg, M. Rubey, personal communication.

\bibitem{JT}
J. Touchard, 
Sur un probl\`eme de configurations et sur les fractions continues,
Can. Jour. Math., Vol. 4 (1952), 2-25.

\bibitem{USW}
M. Uchiyama, T. Sasamoto, M. Wadati, Asymmetric simple exclusion process with 
open boundaries and Askey-Wilson polynomials, J. Phys. A: Math. Gen. 37 (2004),
4985-5002.

\bibitem{AV}
A. Varvak, Rook numbers and the normal ordering problem, 
J. Combin. Theory Ser. A,  Vol. 112(2), (2005), 292-307.

\bibitem{XGV}
X.G. Viennot,  Alternative tableaux and permutations, in preparation (2008).

\bibitem{XGV2}
X.G. Viennot, Alternative tableaux and partially asymmetric exclusion process, in 
preparation (2008).

\bibitem{LW}
L. K. Williams,  Enumeration of totally positive Grassmann cells,
Adv. Math., Vol. 190(2), (2005), 319-342.

\end{thebibliography}
\end{document}